\documentclass[12pt]{amsart}

\usepackage{hyperref}
\usepackage{amsmath}
\usepackage{amssymb}
\usepackage{mathrsfs}
\usepackage{amsthm}
\usepackage[utf8]{inputenc}
\usepackage{svninfo}
\usepackage{prettyref}

\makeatletter

\ifx\Presentation\undefined
\fi

\newcommand{\arXiv}[1]{\href{http://arxiv.org/abs/#1}{arXiv:#1}}

\newcommand{\fref}[1]{\prettyref{#1}}
\newrefformat{item}{\ref{#1}}
\newrefformat{apx}{Appendix~\ref{#1}}

\ifx\thmnum\undefined
\newtheorem{thmnum}{}[section]
\fi

\newcommand{\mynewthm}[3][]{%
  \def\PARAM{#1}
  \ifx\PARAM\empty
  \newtheorem{#2}[thmnum]{#3}
  \else
  \newtheorem{#2}{#3}[#1]
  \fi
  \newtheorem*{#2*}{#3}%
  \newrefformat{#2}{#3~\ref{##1}}%
}

\ifx\undefined\undefined
\newcommand{\ThmLabel}{Theorem}
\newcommand{\PrpLabel}{Proposition}
\newcommand{\LemLabel}{Lemma}
\newcommand{\FctLabel}{Fact}
\newcommand{\CorLabel}{Corollary}
\newcommand{\DfnLabel}{Definition}
\newcommand{\ConvLabel}{Convention}
\newcommand{\NtnLabel}{Notation}
\newcommand{\CstLabel}{Construction}
\newcommand{\ExmLabel}{Example}
\newcommand{\RmkLabel}{Remark}
\newcommand{\QstLabel}{Question}

\else
\newcommand{\ThmLabel}{\iflanguage{french}{Théorème}{Theorem}}
\newcommand{\PrpLabel}{Proposition}
\newcommand{\LemLabel}{\iflanguage{french}{Lemme}{Lemma}}
\newcommand{\FctLabel}{\iflanguage{french}{Fait}{Fact}}
\newcommand{\CorLabel}{\iflanguage{french}{Corollaire}{Corollary}}
\newcommand{\DfnLabel}{\iflanguage{french}{Définition}{Definition}}
\newcommand{\ConvLabel}{Convention}
\newcommand{\NtnLabel}{Notation}
\newcommand{\CstLabel}{Construction}
\newcommand{\ExmLabel}{\iflanguage{french}{Exemple}{Example}}
\newcommand{\RmkLabel}{\iflanguage{french}{Remarque}{Remark}}
\newcommand{\QstLabel}{Question}

\fi

\theoremstyle{plain}
\mynewthm{thm}{\ThmLabel}
\mynewthm{prp}{\PrpLabel}
\mynewthm{lem}{\LemLabel}
\mynewthm{fct}{\FctLabel}
\mynewthm{cor}{\CorLabel}

\theoremstyle{definition}
\mynewthm{dfn}{\DfnLabel}
\mynewthm{conv}{\ConvLabel}
\mynewthm{conj}{Conjecture}
\mynewthm{ntn}{\NtnLabel}
\mynewthm{cst}{\CstLabel}

\theoremstyle{remark}
\mynewthm{rmk}{\RmkLabel}
\mynewthm{qst}{\QstLabel}
\mynewthm{exm}{\ExmLabel}

\newtheorem*{clm}{Claim}
\newenvironment{clmprf}{%
  \begin{proof}[Proof of claim]%
  }{\end{proof}}

\ifx\Presentation\undefined

\newcommand{\myenumlabel}[1]{\textnormal{(\roman{#1})}}
\renewcommand{\theenumi}{\myenumlabel{enumi}}
\fi

\renewcommand{\today}{%
  \number\day\space
  \ifcase\month\or
  January\or February\or March\or April\or May\or June\or
  July\or August\or September\or October\or November\or December\fi
  \space \number\year}

\newcounter{cycprfcnt}
{\begin{list}{\PackageWarning{begnac}{Label required for cycprf}}%
  {%
    \setcounter{cycprfcnt}{1}
    \setlength{\itemindent}{0.5\leftmargin}%
    \setlength{\leftmargin}{0pt}%
    \newcommand{\cpcurr}{\myenumlabel{cycprfcnt}}%
    \newcommand{\cpnext}{\addtocounter{cycprfcnt}{1}\cpcurr}%
    \newcommand{\cpprev}{\addtocounter{cycprfcnt}{-1}\cpcurr}%
    \newcommand{\cpnum}[1]{\setcounter{cycprfcnt}{##1}\cpcurr}%
    \newcommand{\cpfirst}{\cpnum{1}}%
  }%
}%
{\qedhere\end{list}}%

\def\indsym#1#2{%
  \setbox0=\hbox{$\m@th#1x$}%
  \kern\wd0%
  \hbox to 0pt{\hss$\m@th#1\mid$\hbox to 0pt{$\m@th#1^{#2}$\hss}\hss}%
  \lower.9\ht0\hbox to 0pt{\hss$\m@th#1\smile$\hss}%
  \kern\wd0}

\def\nindsym#1#2{%
  \setbox0=\hbox{$\m@th#1x$}%
  \kern\wd0%
  \hbox to 0pt{\hss$\m@th#1\not$\kern1.4\wd0\hss}
  \hbox to 0pt{\hss$\m@th#1\mid$\hbox to 0pt{$\m@th#1^{#2}$\hss}\hss}%
  \lower.9\ht0\hbox to 0pt{\hss$\m@th#1\smile$\hss}%
  \kern\wd0}

\def\dotminussym#1#2{%
  \setbox0=\hbox{$\m@th#1-$}%
  \kern.5\wd0%
  \hbox to 0pt{\hss\hbox{$\m@th#1-$}\hss}%
  \raise.6\ht0\hbox to 0pt{\hss$\m@th#1.$\hss}%
  \kern.5\wd0}
\newcommand{\dotminus}{\mathbin{\mathpalette\dotminussym{}}}

\renewcommand{\emptyset}{\varnothing}
\renewcommand{\setminus}{\smallsetminus}
\def\models{\vDash}

\newcommand{\rest}{{\restriction}}
\newcommand{\flim}{\mathop{\mathcal{F}\mathrm{lim}}}
\newcommand{\half}[1][1]{\hbox{$\frac{#1}{2}$}}

\DeclareMathOperator{\tp}{tp}

\DeclareMathOperator{\Th}{Th}
\DeclareMathOperator{\Mod}{Mod}
\DeclareMathOperator{\tS}{S}

\DeclareMathOperator{\id}{id}

\DeclareMathOperator{\Pert}{Pert}

\DeclareMathOperator{\Diag}{Diag}

\DeclareMathOperator{\sgn}{sgn}

\newcommand{\fB}{\mathfrak{B}}
\newcommand{\fC}{\mathfrak{C}}

\newcommand{\fp}{\mathfrak{p}}

\newcommand{\cA}{\mathcal{A}}

\newcommand{\cC}{\mathcal{C}}

\newcommand{\cL}{\mathcal{L}}
\newcommand{\cK}{\mathcal{K}}

\newcommand{\cN}{\mathcal{N}}

\newcommand{\cX}{\mathcal{X}}

\newcommand{\sA}{\mathscr{A}}

\newcommand{\sE}{\mathscr{E}}

\newcommand{\sU}{\mathscr{U}}

\newcommand{\bN}{\mathbb{N}}

\newcommand{\bQ}{\mathbb{Q}}
\newcommand{\bR}{\mathbb{R}}
\newcommand{\bZ}{\mathbb{Z}}

\makeatother

\DeclareMathOperator{\essrng}{ess\ rng}

\begin{document}

\title{Modular functionals and perturbations of Nakano spaces}

\author{Itaï \textsc{Ben Yaacov}}

\address{Itaï \textsc{Ben Yaacov} \\
  Université Claude Bernard -- Lyon 1 \\
  Institut Camille Jordan \\
  43 boulevard du 11 novembre 1918 \\
  69622 Villeurbanne Cedex \\
  France}

\urladdr{\url{http://math.univ-lyon1.fr/~begnac/}}

\thanks{Research initiated during the workshop ``Model theory of
  metric structures'', American Institute of Mathematics Research
  Conference Centre, 18 to 22 September 2006}
\thanks{Research supported by
  ANR chaire d'excellence junior THEMODMET (ANR-06-CEXC-007) and
  by Marie Curie research network ModNet}

\svnInfo $Id: Nakano.tex 828 2009-02-21 11:30:21Z begnac $
\thanks{\textit{Revision} {\svnInfoRevision} \textit{of} \today}


\begin{abstract}
  We settle several questions regarding the model theory of Nakano
  spaces left open by the PhD thesis of Pedro Poitevin
  \cite{Poitevin:PhD}.

  We start by studying isometric Banach lattice embeddings of Nakano
  spaces, showing that in dimension two and above such embeddings have
  a particularly simple and rigid form.

  We use this to show show that in the Banach lattice language
  the modular functional is definable and that complete theories of
  atomless Nakano spaces are
  model complete.
  We also show that up to
  arbitrarily small perturbations of the exponent Nakano spaces are
  $\aleph_0$-categorical and $\aleph_0$-stable.
  In particular they are stable.
\end{abstract}

\maketitle

\section*{Introduction}

\emph{Nakano spaces} are a generalisation of $L_p$ function spaces in
which the exponent $p$ is allowed to vary as a measurable function of
the underlying measure space.
The PhD thesis of Pedro Poitevin \cite{Poitevin:PhD} studies Nakano
spaces as Banach lattices from a model theoretic standpoint.
More specifically, he viewed Nakano spaces as continuous metric
structures (in the sense of continuous logic, see
\cite{BenYaacov-Usvyatsov:CFO}) in the language of Banach lattices,
possibly augmented by a predicate symbol $\Theta$ for the modular
functional, showed that natural classes of such structures are
elementary in the sense of continuous first order logic, and studied
properties of their theories.

In the present paper we propose to answer a few questions left open by
Poitevin.
\begin{itemize}
\item First, Poitevin studies Nakano spaces in two natural languages:
  that of Banach lattices, and the same augmented with an additional
  predicate symbol for the modular functional.
  It is natural to ask whether these languages are truly distinct,
  i.e., whether adding the modular functional adds new structure.
\item Even if the naming of the modular functional does not add
  structure, it does give quantifier elimination in atomless Nakano
  spaces.
  While it is clear that without it quantifier elimination is
  impossible, it is natural to ask whether model completeness is.
\item Poitevin showed that the theory of atomless Nakano space where
  the exponent function is \emph{bounded away from one} is stable.
  What about the general case?
\item Similarly,
  if the exponent is constant, i.e., if we are
  dealing with classical atomless $L_p$ spaces, it is known (see
  \cite{BenYaacov-Berenstein-Henson:LpBanachLattices})
  that the theory of these spaces is
  $\aleph_0$-categorical and $\aleph_0$-stable.
  On the other hand, it is quite easy to verify
  complete theories of atomless Nakano spaces are non
  $\aleph_0$-categorical and non $\aleph_0$-stable once the
  essential range of the exponent is infinite.
  It is therefore natural to ask whether, up to small perturbations of
  the exponent, a complete theory of atomless Nakano spaces
  is $\aleph_0$-categorical and $\aleph_0$-stable.
  A positive answer would mean that the theory of atomless Nakano
  spaces is stable settling the previous item as well.
\end{itemize}

In this paper we answer all of these questions positively (where a
negative answer to the first question is considered positive).
It is organised as follows:

\fref{sec:FunctionalAnalysis} consist purely of functional analysis,
and requires no familiarity with model theory.
After a few general definitions we study mappings between vector
lattices of measurable functions and then more specifically between
Nakano spaces.
Our main result is:
\begin{thm*}
  Let
  $\theta\colon L_{p(\cdot)}(X,\fB,\mu)
  \hookrightarrow L_{q(\cdot)}(Y,\fC,\nu)$
  be a
  Banach lattice isometric embedding of Nakano spaces of dimension at
  least two.
  Then up to a measure density change on $Y$ and identification
  between subsets of $X$ and of $Y$ (and thus between measurable
  functions on $X$ and on $Y$) $\theta$ is merely an extension by zeros
  from $X$ to $Y \supseteq X$.
  In particular $p = q\rest_X$ and $\mu = \nu\rest_\fB$.
\end{thm*}
It follows that such embeddings respect the modular functional and
extend the essential range of the exponent function.

In \fref{sec:NakanoModelTheory} we expose the model theoretic setting
for the paper.
In particular, we quote the main results of
Poitevin's PhD thesis \cite{Poitevin:PhD}.

In \fref{sec:DefinableModular} we prove our main model theoretic results:

\begin{thm*}
  The modular functional is definable in every Nakano Banach
  lattice (i.e., naming it in the language does not add structure).
  Moreover, it is uniformly definable in the
  class of Nakano spaces of dimension at least two, and in fact both
  $\sup$-definable and $\inf$-definable there.
\end{thm*}

\begin{thm*}
  The theory of atomless Nakano
  spaces with a fixed essential range for the exponent function is
  model complete in the Banach lattice language.
\end{thm*}

In \fref{sec:PertExp} we study perturbations of the exponent function,
showing that small perturbations thereof yield small perturbations of
the structures.
Up to such perturbations the theory of atomless Nakano spaces is
$\aleph_0$-stable, and every completion thereof is $\aleph_0$-categorical.
In particular all Nakano space are stable.

\fref{apx:ContModTh} consist of the adaptation
to continuous logic of a few classical model theoretic results and
tools used in this paper.

Finally, \fref{apx:ModularApprox} contains some approximation results
for the modular functional which were used in earlier versions of this
paper to be superseded later by \fref{thm:EmbedFactor}, but which
nonetheless might be useful.

\section{Some functional analysis}
\label{sec:FunctionalAnalysis}

\subsection{Nakano spaces}

Let $(X,\fB,\mu)$ be an arbitrary measure space, and let
$L_0(X,\fB,\mu)$ be the space of all measurable functions
$f\colon X \to \bR$ up to equality a.e.
(Since we wish to consider function spaces as Banach lattices it will be
easier to consider the case of real-valued functions.)

Let $p\colon X \to [1,\infty)$ be an
essentially bounded measurable function.
We define the \emph{modular functional}
$\Theta_{p(\cdot)}\colon L_0(X,\fB,\mu) \to [0,\infty]$ by:
\begin{gather*}
  \Theta_{p(\cdot)}(f) = \int |f(x)|^{p(x)} d\mu.
\end{gather*}
We define the corresponding \emph{Nakano space} as:
\begin{gather*}
  L_{p(\cdot)}(X,\fB,\mu)
  = \{f \in L_0(X,\fB,\mu)\colon \Theta_{p(\cdot)}(f) < \infty\}.
\end{gather*}
If $f \in L_{p(\cdot)}(X,\fB,\mu)$ then there exists a
unique number $c \geq 0$ such that $\Theta_{p(\cdot)}(f/c) = 1$, and we
define $\|f\| = \|f\|_{p(\cdot)} = c$.
This is a norm, making $L_{p(\cdot)}(X,\fB,\mu)$ a Banach space.
With the point-wise minimum and maximum operations it is a Banach
lattice.

\begin{rmk*}
  In the literature $\Theta_{p(\cdot)}$ is usually merely referred to as the
  \emph{modular}.
  Being particularly sensitive regarding parts of speech
  we shall nonetheless refer to it throughout as the modular
  \emph{functional}.
\end{rmk*}

\subsection{Strictly localisable spaces}

In this paper we shall consider the class of Nakano spaces from a
model-theoretic point of view.
This means we shall have to admit arbitrarily large Nakano spaces
(e.g., $\kappa$-saturated for arbitrarily big $\kappa$) and therefore
arbitrarily large measure spaces.
In particular, we cannot restrict our attention to $\sigma$-finite measure
spaces.
In order to avoid pathologies which may arise with arbitrary measure
spaces we shall require a weaker assumption.
Recall from \cite{Fremlin:MeasureTheoryVol2}:
\begin{dfn}
  A measure space $(X,\fB,\mu)$ is \emph{strictly localisable}
  if it can be expressed as a disjoint union of measure spaces of
  finite measure, i.e.,
  if $X$ admits a partition as $\bigcup_{i \in I} X_i$ such that:
  \begin{enumerate}
  \item For all $i \in I$: $X_i \in \fB$ and $\mu(X_i) < \infty$.
  \item For all $A \subseteq X$:
    $A \in \fB$ if and only if $A \cap X_i \in \fB$ for
    all $i \in I$, in which case $\mu(A) = \sum \mu(A\cap X_i)$.
  \end{enumerate}
  In this case the family $\{X_i\}_{i \in I}$ \emph{witnesses} that
  $(X,\fB,\mu)$ is strictly localisable.
\end{dfn}

For example every $\sigma$-finite measure space is strictly localisable.
On the other hand, if $(X,\fB,\mu)$ is an arbitrary measure space we can
find a maximal family $\cX = \{X_i\}_{i \in I} \subseteq \fB$
of almost disjoint
sets with $0 < \mu(X_i) < \infty$.
Let
$(X',\fB',\mu')
= \coprod_{i\in I} (X_i,\fB\rest_{X_i},\mu\rest_{X_i})$,
where the disjoint union of measure spaces is defined precisely so
that the result is strictly localisable.
We also have an
obvious mapping $\theta\colon L_0(X,\fB,\mu) \to L_0(X',\fB',\mu')$.
This does not lose any information that interests us:
in particular, $\theta$ restricts to an isometric isomorphism of Nakano
spaces
$\theta\colon L_{p(\cdot)}(X,\fB,\mu)
\to L_{\theta p(\cdot)}(X',\fB',\mu')$.

We may therefore allow ourselves:
\begin{conv}
  \label{conv:StrLoc}
  In this paper every measure space
  is assumed to be strictly localisable.
\end{conv}

Let us state a few very easy facts concerning strictly localisable
measure spaces.
The following is immediate:
\begin{fct}
  \label{fct:WitnessRefinement}
  Let $\cX = \{X_i\}_{i \in I}$ witness that $(X,\fB,\mu)$ is strictly
  localisable.
  If $\cX' = \{X'_j\}_{j \in J} \subseteq \fB$
  is another partition of $X$ refining
  $\cX$, splitting each $X_i$ into at most countably many subsets,
  then $\cX'$ is a witness as well.
\end{fct}

The Radon-Nikodým Theorem is classically stated for finite measure
spaces, with various occurrences in the literature
in which the finiteness requirement on the ambient space
is relaxed.
See for example \cite[Corollaries~232F,G]{Fremlin:MeasureTheoryVol2}.
These are corollaries to
\cite[Theorem~232E]{Fremlin:MeasureTheoryVol2}, which allows an
arbitrary ambient measure space at the cost of an additional
concept, that of a \emph{truly continuous} functional.
Another generalisation appears in
\cite[Theorem~327D]{Fremlin:MeasureTheoryVol3},
but again the smaller measure is assumed there to be finite.
We shall require a different generalisation of
the Radon-Nikodým Theorem in which
all finiteness requirements are replaced with
strict localisability.

Let $(X,\fB)$ be a measurable space and let $\mu$ and $\nu$ be two
measures on $(X,\fB)$.
Assume also that $\nu(X) < \infty$.
Then
$\nu$ is said to be \emph{absolutely continuous} with
respect to $\mu$, in symbols $\nu \ll \mu$,
if for every $\varepsilon > 0$ there exists $\delta > 0$
such that
$\mu(A) < \delta \Longrightarrow \nu(A) < \varepsilon$
for every $A \in \fB$.
Equivalently, if
$\mu(A) = 0 \Longrightarrow \nu(A) = 0$
for every $A \in \fB$.

In the general case, i.e., when $\nu$ is not required to be finite,
we shall use the notation $\nu \ll \mu$ to mean that
$\mu$ and $\nu$ are both strictly localisable with a common
witness $\{X_i\}_{i \in I}$, and that
$\nu$ is absolutely continuous with respect to $\mu$ on each
$X_i$.
It follows directly from this definition that if
$\nu \ll \mu$ and $\mu(A) = 0$ for some $A \in \fB$ then
$\nu(A) = 0$ as well, so $\nu$ is absolutely continuous with respect
to $\mu$ on every set of finite $\nu$-measure.
This has two important consequences.
First, if $\nu \ll \mu$ then every common witness of strict
localisability for both $\mu$ and $\nu$ also witnesses that
$\nu \ll \mu$.
Second, in case $\nu(X)  < \infty$, the definition
of $\nu \ll \mu$ given in this paragraph coincides with
the classical definition appearing in the previous paragraph.

\begin{fct}
  \label{fct:RadonNikodym}
  Let $(X,\fB,\mu)$ be a measure space
  (strictly localisable, by our convention)
  and let $L_0^+(X,\fB,\mu)$ denote the set of positive functions
  in $L_0(X,\fB,\mu)$.
  \begin{enumerate}
  \item Let $\zeta \in L_0^+(X,\fB,\mu)$,
    and for $A \in \fB$ define
    $\nu_\zeta(A) = \int \zeta\, d\mu$.
    Then $\nu_\zeta$ is a measure and $\nu_\zeta \ll \mu$.
  \item Conversely, every measure $\nu \ll \mu$ on $(X,\fB)$
    is of the form $\nu = \nu_\zeta$ for a unique
    (up to equality $\mu$-a.e.)\ $\zeta \in L_0^+(X,\fB,\mu)$,
    and we write $\zeta = \frac{d\nu}{d\mu}$,
    the \emph{Radon-Nikodým derivative} of $\nu$ with respect to
    $\mu$.
    In this case we also have
    $\int f\, d\nu = \int f\frac{d\nu}{d\mu}\, d\mu$
    for every $f \in L_0^+(X,\fB,\mu)$.
  \end{enumerate}
  In particular, we obtain a bijection between
  $\{\nu\colon \nu \ll \mu\}$ and
  $L_0^+(X,\fB,\mu)$.
\end{fct}
\begin{proof}
  For the first item, let $\{X_i\}_{i \in I}$ witness that
  $\mu$ is strictly localisable.
  We may assume that in addition $\zeta$ is bounded on each
  $X_i$, for if not, we may split each $X_i$ into
  $X_{i,n} = \{x \in X_i\colon n \leq \zeta(x) < n+1\}$
  for $n \in \bN$.
  Then $\{X_i\}_{i \in I}$ also witnesses that 
  $\nu_\zeta$ is strictly localisable and it is clear that
  $\nu_\zeta \ll \mu$.

  For the converse,
  let $\{X_i\}_{i \in I}$ witness that $\nu \ll \mu$.
  We may apply the classical Radon-Nikodým theorem on each
  $X_i$, obtaining a measurable function
  $\zeta_i\colon X_i \to \bR^+$ for all $i \in I$,
  and define $\zeta\colon X \to \bR^+$ so that
  $\zeta\rest_{X_i} = \zeta_i$.
  Then $\zeta$ is measurable and
  \begin{gather*}
    \int f\,d\nu
    = \sum \int_{X_i} f\, d\nu
    = \sum \int_{X_i} f\zeta_i\, d\mu
    = \int f\zeta\, d\mu
  \end{gather*}
  for $f \in L_0^+(X,\fB,\mu)$.
  In particular
  $\nu(A) = \int_A \zeta d\mu$
  for all $A \in \fB$,
  which determines $\zeta$ up to equality $\mu$-a.e.
\end{proof}

Let us say that two measures $\mu$ and $\nu$ on $(X,\fB)$
are \emph{equivalent} if $\mu \ll \nu$ and $\nu \ll \mu$.
In this case each is obtained from the
other by a mere \emph{density change} and
the corresponding Nakano spaces are naturally isomorphic.
\begin{fct}
  \label{fct:DenseChange}
  Let $\mu$ and $\nu$ be two equivalent measures on $(X,\fB)$,
  and let  $p\colon X \to [1,r]$ be measurable.
  Let $(N,\Theta) =  L_{p(\cdot)}(X,\fB,\mu)$ and
  $(N',\Theta') = L_{p(\cdot)}(X,\fB,\nu)$
  be the corresponding Nakano spaces with their modular functionals.
  For $f \in N$ define $D_{\mu,\nu}f = (\frac{d\mu}{d\nu})^{1/p}f$.
  Then $D_{\mu,\nu}f \in N'$ and
  $D_{\mu,\nu}\colon (N,\Theta) \simeq (N',\Theta')$
  is an (isometric) isomorphism.
\end{fct}
\begin{proof}
  One calculates:
  \begin{align*}
    \Theta'(D_{\mu,\nu}f)
    & = \int \left( (\hbox{$\frac{d\mu}{d\nu}$})^{1/p}|f| \right)^p\, d\nu \\
    & = \int |f|^p \hbox{$\frac{d\mu}{d\nu}$}\,d\nu \\
    & = \int |f|^p\, d\mu = \Theta(f).
  \end{align*}
  It follows that $f \in N \Longrightarrow D_{\mu,\nu} f \in N'$.
  In addition to $\Theta$, $D_{\mu,\nu}$ clearly also
  respects the linear and lattice structures, and therefore the norm,
  and admits an inverse $D_{\nu,\mu}$.
\end{proof}

\subsection{Mappings between function space lattices}

For the following results we shall be considering two measure spaces
$(X,\fB,\mu)$ and $(Y,\fC,\nu)$, as well as a partial mapping
$\theta\colon L_0(X,\fB,\mu) \dashrightarrow L_0(Y,\fC,\nu)$.
Its domain $L \subseteq L_0(X,\fB,\mu)$ is a vector subspace
which contains all characteristic functions of finite measure sets.
For example, $L$ could be a Nakano space $L_{p(\cdot)}(X,\fB,\mu)$
or just the space of simple functions on $X$ with finite measure
support.
Assuming that $\theta$ sends characteristic functions to
characteristic functions, we shall allow ourselves the following
abuse of notation:
if $A \in \fB$ has finite measure and $\theta(\chi_A) = \chi_B$,
$B \in \fC$ then we write $\theta A = B$
(even though this is only defined up to null measure).
In particular, instead of writing $\theta(\chi_A)$ we write
$\chi_{\theta A}$.

\begin{lem}
  \label{lem:SetMorphism}
  Let $L \subseteq L_0(X,\fB,\mu)$ be a vector subspace
  which contains all characteristic functions of finite measure sets
  and let
  $\theta\colon L \to L_0(Y,\fC,\nu)$
  be a linear mapping respecting point-wise
  countable suprema when those exist
  in $L$, and which in addition sends characteristic functions to
  characteristic functions.

  Then $\theta$ extends to a unique
  vector lattice homomorphism
  $\hat \theta\colon L_0(X,\fB,\mu) \to L_0(Y,\fC,\nu)$ which respects countable
  suprema.
  Moreover, for every Borel function
  $\varphi\colon \bR^n \to \bR$ which fixes zero (i.e., which sends
  $0 \in \bR^n$ to $0 \in \bR$) and every tuple $\bar f \in L_0(X,\fB,\mu)$ we
  have
  $\hat \theta(\varphi \circ \bar f)
  = \varphi \circ (\theta \bar f)$.
\end{lem}
\begin{proof}
  Let us write $L_0$ for $L_0(X,\fB,\mu)$, and let
  $L_0^+$ be its positive cone.

  Let us first consider the case where $\mu(X),\nu(Y) < \infty$.
  In this case $L$ contains all simple measurable
  functions.
  For $f \in L_0^+$ and $0 < k \in \bN$
  define $f^{(k)}(x) = \frac{\lceil kf(x) \rceil}{k+1} \wedge k$,
  where $\lceil r \rceil$ denotes the least integer greater than $r$.
  Thus $f^{(k)} \nearrow f$ point-wise and $f^{(k)} \in L$ for all
  $k$.
  We then have no choice but to define $\hat \theta$ as follows:
  \begin{align*}
    & \hat \theta f
    = \hat \theta\left( \bigvee_{k\in \bN} f^{(k)} \right)
    = \bigvee_{k\in \bN} \theta f^{(k)}
    && \text{for } f \in L_0^+, \\
    & \hat \theta f
    = \hat \theta(f^+ - f^-)
    = \hat \theta f^+ - \hat \theta f^-
    &&  \text{for general } f \in L_0.
  \end{align*}
  We now need to make sure this verifies all the requirements.

  First of all we need to check that if $f \in L_0^+$ then
  $\hat \theta f = \bigvee_{k\in \bN} \theta f^{(k)}$
  exists, i.e., that it is finite a.e.
  Let
  $A_k = \{f \geq k\} = \{x \in X\colon f(x)\geq k\}$.
  Then the sequence $\{\chi_{A_k}\}$ decreases to zero,
  whereby $\{\chi_{\theta A_k}\}$ decrease to zero as well.
  We have
  $f^{(k+m)} \leq k + m\chi_{A_k}$ whereby
  $\theta f^{(k+m)} \leq k + m\chi_{\theta A_k}$,
  so $\theta f^{(k+m)} \leq k$ outside $\theta A_k$,
  for all $m$.
  Thus $\hat \theta f \leq k$  outside $\theta A_k$,
  and we can conclude that $\hat \theta f \in L_0(Y,\fC,\nu)$.
  Since $\theta$ respects countable suprema, $\hat \theta$
  agrees with $\theta$ on $L^+$.

  We claim that $\hat \theta$ respects countable suprema
  on $L_0^+$.
  Indeed, assume that $\bigvee_{m\in\bN} f_m$ exists for
  $f_m \in L_0^+$.
  Notice that in general
  $\bigvee_m \lceil a_m \rceil =
  \left\lceil \bigvee_m a_m \right\rceil$, whereby
  \begin{gather*}
    \hat \theta\left( \bigvee_{m\in\bN} f_m \right)
    = \bigvee_{k\in\bN} \theta\left( \left(
        \bigvee_{m\in\bN} f_m
      \right)^{(k)} \right)
    = \bigvee_{k\in\bN} \theta\left(
        \bigvee_{m\in\bN} f_m^{(k)}
      \right)
    = \bigvee_{m\in\bN,k\in\bN} \theta(f_m^{(k)})
    = \bigvee_{m \in \bN} \hat\theta(f_m).
  \end{gather*}
  If $f = \sum_{m \in \bN} f_m$ where $f_m \in L_0^+$, $f_m\wedge f_{m'} = 0$
  for $m \neq m'$ then $\theta(f_m) \wedge \theta(f_{m'}) = 0$ as well
  and
  $\hat \theta(f)
  = \hat\theta(\bigvee_mf_m)
  = \bigvee_m\hat\theta(f_m)
  = \sum_m\hat\theta(f_m)$.

  Next we claim that if $A \subseteq (0,\infty)^n$ is a Borel set and
  $\bar f \in (L_0^+)^n$ then
  $\theta\{\bar f(x) \in A\} = \{\hat \theta\bar f(y) \in A\}$.
  Indeed, for a single $f$ we have
  $\hat\theta f
  = \hat\theta(\bigvee_{t\in\bQ^+}t\chi_{\{f>t\}}) 
  = \bigvee_{t\in\bQ^+}t\chi_{\theta\{f>t\}}$, whereby
  $\{\hat\theta f>t\} = \theta\{f>t\}$.
  Our claim follows for the case
  $A = (t_0,\infty) \times \ldots \times (t_{n-1},\infty)$.
  On the other hand we have
  \begin{gather*}
    \theta\left\{ \bar f(x) \in \bigcup_m A_m \right\}
    = \bigcup_m \theta\{\bar f(x) \in A_m\}, \\
    \left\{ \hat \theta\bar f(y) \in \bigcup_m A_m \right\}
    = \bigcup_m \{\hat \theta\bar f(y) \in A_m\}, \\
    \theta\{\bar f(x) \in (0,\infty)^n \setminus A\}
    = \theta \{\bar f >0 \} \setminus \theta\{\bar f(x) \in A\}, \\
    \{\hat \theta\bar f(y) \in (0,\infty)^n \setminus A\}
    = \theta \{\bar f >0 \} \setminus \{\hat \theta\bar f(y) \in A\}.
  \end{gather*}
  We may thus climb up the Borel hierarchy and prove the claim
  for all Borel $A$.

  Assume now that $\bar f(x) \in (0,\infty)^n\cup\{0\}$
  for all $x \in X$ and that $\varphi \geq 0$.
  Letting $A_t= \{x\in(0,\infty)^n\colon \varphi(x) > t\}$:
  \begin{gather*}
    \theta\{\varphi \circ \bar f > t\}
    =  \theta\{\bar f \in A_t\}
    =  \{\hat\theta\bar f \in A_t\}
    =  \{\varphi \circ (\hat\theta\bar f) > t\},
  \end{gather*}
  whereby
  $\hat \theta(\varphi \circ \bar f)
  = \varphi \circ (\hat\theta \bar f)$.
  For general $\bar f$,
  let $S = \{1,0,-1\}^n \setminus \{0\}$, and for $s \in S$ let
  $A_s = \{x\in X\colon \sgn (\bar f) = s\}$.
  On each $A_s$ we may drop those $f_i$'s which are constantly zero and
  replace those which are negative with their absolute value, making
  the necessary modifications to $\varphi$, obtaining by the previous
  argument
  \begin{gather*}
    \hat \theta(\varphi \circ (\chi_{A_s}\bar f))
    = \varphi \circ (\hat \theta(\chi_{A_s}\bar f)),
    \intertext{whereby:}
    \hat\theta(\varphi \circ \bar f)
    = \sum_{s \in S} \hat\theta(\varphi \circ (\chi_{A_s}\bar f))
    = \sum_{s \in S} \varphi \circ (\hat\theta(\chi_{A_s}\bar f))
    = \varphi \circ (\hat\theta \bar f)
  \end{gather*}
  Finally, for general $\varphi$ we can split it to the positive and negative
  part and then put them back together by linearity.
  Among other things, this holds when $\varphi$
  is $+$, $\vee$, $\wedge$, or multiplication by scalar.
  Thus $\hat \theta$ is a vector lattice homomorphism.
  It follows that $\hat \theta$ agrees with $\theta$ on all of $L$.
  This concludes the case where both $X$ and $Y$ have finite total
  measure.

  Now let us consider the case where $X$ is an arbitrary measure
  space.
  Let $\{X_i\}_{i\in I} \subseteq \fB$ be a maximal family of almost
  disjoint sets of finite non zero measure such that in addition
  $\theta(\chi_{X_i}) \neq 0$.
  Since $\nu(Y)$ is assumed finite such a family must be at most
  countable, so we can write it as $\{X_k\}_{k\in\bN}$.
  Let $X' = \bigcup X_k$.
  Then for every $f \in L$ we have
  $\theta(f) = \theta(f\chi_{X'}) = \sum_k \theta(f\chi_{X_k})$
  (verify first for $f\geq0$ and then extend by linearity),
  so we may restrict to each
  $X_k$, reducing to the case already considered, then checking that
  $\hat\theta(f) = \sum_k \hat\theta(f\chi_{X_k})$ works.

  Finally, if $(Y,\fC,\nu)$ is merely strictly localisable then
  let this be witnessed by $\{Y_i\}_{i \in I}$.
  Then we can first extend
  $\theta_i = \chi_{Y_i}\theta \colon L \to (Y_i,\fC\rest_{Y_i},\nu\rest_{Y_i})$
  to $\hat \theta_i$ and then obtain $\hat\theta$ by gluing.
\end{proof}

\begin{lem}
  \label{lem:SetMorphismIntegral}
  Continue with previous assumptions, and add that
  if $\mu(A) < \infty$ then $\nu(\theta A) = \mu(A)$.
  Then for every function
  $f \in L_1(X,\fB,\mu)$: $\int f\, d\mu = \int \hat\theta f\,d\nu$.
\end{lem}
\begin{proof}
  This holds by assumption for characteristic functions of finite
  measure sets, from which we deduce it for simple positive functions,
  positive functions and finally general functions.
\end{proof}

\subsection{Embeddings of Nakano Banach lattices}

We now prove the main functional analysis results of this paper.

\begin{lem}
  \label{lem:NakanoEmbedSets}
  Let $N = L_{p(\cdot)}(X,\fB,\mu)$ and $N' = L_{q(\cdot)}(Y,\fC,\nu)$ be two
  Nakano spaces,
  and let $\theta\colon N \to N'$ be an isometric
  embedding of Banach lattices which sends
  characteristic functions to characteristic functions.
  Assume furthermore that $\dim N \geq 2$.
  Then:
  \begin{enumerate}
  \item $\hat\theta(p) = q\chi_{\hat\theta X}$.
  \item For all finite measure $A \in \fB$:
    $\nu(\theta A) = \mu(A)$.
  \end{enumerate}
\end{lem}
\begin{proof}
  First of all the hypotheses of \fref{lem:SetMorphism}
  are verified with $N = L$, so $\hat \theta$ exists.
  Let $Y_0  = \hat\theta X \in \fC$ be the support of the range of $\theta$.
  \begin{align*}
    C_1 & = \{y\in Y_0\colon \hat\theta p(y) < q(y)\}, \\
    C_2 & = \{y\in Y_0\colon \hat\theta p(y) > q(y)\}, \\
    C & = C_1 \cup C_2 = \{y\in Y_0\colon \hat\theta p(y) \neq q(y)\}
  \end{align*}
  Then $C_1,C_2,C \in \fC$ and we need to show that
  $\nu(C) = 0$.
  Let $A,B \in \fB$ be such that $0 < \mu(A),\mu(B) < \infty$.
  For $t \in [0,1]$, let
  \begin{gather*}
    f_t = \chi_A\left( \frac{t}{\mu(A)} \right)^{\frac{1}{p(x)}} +
    \chi_B\left( \frac{1-t}{\mu(B)} \right)^{\frac{1}{p(x)}}, \\
    g_t = \theta(f_t) = \chi_{\theta A}\left( \frac{t}{\mu(A)} \right)^{\frac{1}{\hat\theta p(y)}} +
    \chi_{\theta B}\left( \frac{1-t}{\mu(B)} \right)^{\frac{1}{\hat\theta p(y)}}, \\
  \end{gather*}
  Then $\Theta(f_t) = 1 \Longrightarrow \|f_t\| = 1 \Longrightarrow \|g_t\| = 1 \Longrightarrow \Theta'(g_t) = 1$.
  In other words:
  \begin{gather*}
    \Theta'(g_t) = \int_{\theta A} \left( \frac{t}{\mu(A)} \right)^{\frac{q}{\hat\theta p}} \, d\nu +
    \int_{\theta B} \left( \frac{1-t}{\mu(B)} \right)^{\frac{q}{\hat\theta p}} \, d\nu = 1
  \end{gather*}
  Substituting $t = 0$ and $t = 1$ we see that in particular
  $\nu(A)$ and $\nu(B)$ are both positive and finite.
  We may therefore differentiate under the integral sign for
  $t \in (0,1)$, obtaining:
  \begin{align*}
    0 = \frac{d}{dt} \Theta'(g_t)
    & =
    \int_{\theta A\cap C} {\frac{q}{\mu(A)\hat\theta p}}
    \left( \frac{t}{\mu(A)} \right)^{\frac{q}{\hat\theta p}-1} \, d\nu
    + \int_{\theta A\setminus C} {\frac{q}{\mu(A)\hat\theta p}} \, d\nu \\
    & \qquad
    -
    \int_{\theta B\cap C} {\frac{q}{\mu(B)\hat\theta p}}
    \left( \frac{1-t}{\mu(B)} \right)^{\frac{q}{\hat\theta p}-1} \, d\nu
    - \int_{\theta B\setminus C} {\frac{q}{\mu(B)\hat\theta p}} \, d\nu
  \end{align*}
  If $\nu(\theta A\cap C_2) > 0$
  then $\lim_{t\to0} \frac{d}{dt} \Theta'(g_t) = +\infty \neq 0$ which is impossible,
  so $\nu(\theta A\cap C_2) = 0$, and considering $t \to 1$ we see that
  $\nu(\theta B\cap C_2) = 0$ as well.
  We may therefore substitute $t = 0$ and $t = 1$ and obtain:
  \begin{align*}
    0
    & =
    \int_{\theta A\setminus C} {\frac{q}{\mu(A)\hat\theta p}} \, d\nu \\
    & \qquad
    -
    \int_{\theta B\cap C} {\frac{q}{\mu(B)\hat\theta p}}
    \left( \frac{1}{\mu(B)} \right)^{\frac{q}{\hat\theta p}-1} \, d\nu
    - \int_{\theta B\setminus C} {\frac{q}{\mu(B)\hat\theta p}} \, d\nu \\
    & =
    \int_{\theta A\cap C} {\frac{q}{\mu(A)\hat\theta p}}
    \left( \frac{1}{\mu(A)} \right)^{\frac{q}{\hat\theta p}-1} \, d\nu
    + \int_{\theta A\setminus C} {\frac{q}{\mu(A)\hat\theta p}} \, d\nu \\
    & \qquad
    - \int_{\theta B\setminus C} {\frac{q}{\mu(B)\hat\theta p}} \, d\nu,
  \end{align*}
  whereby
  \begin{gather*}
    \int_{\theta A\cap C} {\frac{q}{\mu(A)\hat\theta p}}
    \left( \frac{1}{\mu(A)} \right)^{\frac{q}{\hat\theta p}-1} \, d\nu
    =
    -
    \int_{\theta B\cap C} {\frac{q}{\mu(B)\hat\theta p}}
    \left( \frac{1}{\mu(B)} \right)^{\frac{q}{\hat\theta p}-1} \, d\nu
  \end{gather*}
  This is only possible if both are zero, i.e., if
  $\nu(\theta A\cap C) = \nu(\theta B\cap C) = 0$.

  We have shown that $\nu(\theta A\cap C) = \nu(\theta B\cap C) = 0$ for every
  $A,B \in \fB$ disjoint of finite non zero measure.
  If $N$ had dimension $\leq1$ this would be vacuous, but as we assume
  that it has dimension $\geq 2$ we have in fact
  $\nu(\theta A\cap C) = 0$ for all $A \in \fB$ such that $\mu(A) < \infty$.
  It follows that $\nu(C) = \nu(Y_0\cap C) = 0$,
  i.e., that $\hat\theta p = q\chi_{Y_0}$.

  Now let $A \in \fB$ be of finite non zero measure,
  $h = \mu(A)^{-1/p(x)}$.
  Then $\Theta(h) = 1 \Longrightarrow 1 = \Theta'(\theta(h)) = \nu(\theta A)/\mu(A)$.
\end{proof}

\begin{rmk*}
  A special case of this result was independently obtained at the same
  time by Poitevin and Raynaud
  \cite[Lemma~6.1]{Poitevin-Raynaud:PositiveContractiveProjections}.
\end{rmk*}

The technical assumption that $\theta$ sends characteristic functions to
such (i.e., acts on measurable sets) is easy to obtain via a
density change:
\begin{lem}
  \label{lem:NakanoEmbedDenseChange}
  Let $N = L_{p(\cdot)}(X,\fB,\mu)$ and $N' = L_{q(\cdot)}(Y,\fC,\nu)$ be two
  Nakano spaces, and let
  $\theta\colon N \to N'$ be an isometric embedding of Banach lattices.
  Then there is a measure $\lambda$ on $(Y,\fC)$, equivalent to
  $\nu$, such that
  $D_{\nu,\lambda} \circ \theta\colon N \to N'' = L_{q(\cdot)}(Y,\fC,\lambda)$
  sends characteristic functions to
  characteristic functions,
  where $D_{\nu,\lambda}\colon N' \to N''$ is the density change
  isomorphism from \fref{fct:DenseChange}.
\end{lem}
\begin{proof}
  Let $\{X_i\}_{i \in I} \subseteq \fB$ and
  $\{Y_j\}_{j \in J} \subseteq \fC$
  witness that $X$ and $Y$ are strictly localisable.
  Possibly replacing them with refinements as in
  \fref{fct:WitnessRefinement} we may assume that
  $I \subseteq J$ and that for $i \in I$ the set $Y_i$ is
  the support of $\theta(\chi_{X_i})$.
  Define $\zeta\colon Y \to \bR^+$ by
  \begin{gather*}
    \zeta =
    \sum_{i \in I} \theta(\chi_{X_i})^{q}
    + \sum_{j \in J\setminus I}\chi_{Y_j}.
  \end{gather*}
  This function is measurable and non zero a.e., allowing us to define
  another measure $\lambda$ by $d\lambda = \zeta d\nu$.
  Then $\nu$ and $\lambda$ are equivalent measures,
  and $D_{\nu,\lambda} \circ \theta(\chi_{X_i}) = \chi_{Y_i}$.
  Since this is an embedding of Banach lattices it follows that it sends
  every characteristic function to a characteristic function.
\end{proof}

Putting everything together we obtain:
\begin{thm}
  \label{thm:EmbedFactor}
  Let $N = L_{p(\cdot)}(X,\fB,\mu)$ and $N' = L_{q(\cdot)}(Y,\fC,\nu)$ be two
  Nakano spaces, $\dim N \geq 2$, and let
  $\theta\colon N \to N'$ be an isometric embedding of Banach lattices.
  Then up to a measure density change on $Y$:
  \begin{enumerate}
  \item $\theta$ sends characteristic functions to such.
  \item $\hat\theta p = q\chi_{\hat\theta X}$.
  \item For all finite measure $A$: $\nu(\theta A) = \mu(A)$.
  \end{enumerate}
\end{thm}
\begin{proof}
  Immediate from \fref{lem:NakanoEmbedDenseChange}
  and \fref{lem:NakanoEmbedSets}.
\end{proof}

\begin{cor}
  \label{cor:RigidModular}
  Let $(N,\Theta) = L_{p(\cdot)}(X,\fB,\mu)$, $(N',\Theta') = L_{q(\cdot)}(Y,\fC,\nu)$
  be two Nakano spaces, $\dim N \geq 2$, and let
  $\theta\colon N \to N'$ be an embedding of Banach lattices.
  Then $\theta$ respects the modular functional: $\Theta = \Theta' \circ \theta$.
\end{cor}
\begin{proof}
  According to \fref{fct:DenseChange} a density change on $Y$
  does not alter $\Theta'$.
  Thus we may assume that $\theta$ is as in the conclusion of
  \fref{thm:EmbedFactor}.
  By \fref{lem:SetMorphismIntegral} we then obtain
  for all $f \in N$:
  \begin{gather*}
    \Theta'\circ\theta(f) = \int |\theta(f)|^q\, d\nu = \int |\theta(f)|^{\hat\theta p}\, d\nu
    = \int \hat\theta( |f|^p )\,d\nu = \int |f|^p \, d\mu = \Theta(f).
  \end{gather*}
\end{proof}

\begin{cor}
  \label{cor:RigidEssRng}
  Let $(N,\Theta) = L_{p(\cdot)}(X,\fB,\mu)$ and
  $(N',\Theta') = L_{q(\cdot)}(Y,\fC,\nu)$
  be two Nakano spaces, $\dim N \geq 2$, and let
  $\theta\colon N \to N'$ be an embedding of Banach lattices.
  Then $\essrng p \subseteq \essrng q$.
  If the band generated by $\theta(N)$ in $N'$ is all of
  $N'$ (so in particular, if $\theta$ is an isomorphism)
  then $\essrng p  = \essrng q$.
\end{cor}
\begin{proof}
  The density change does not modify $p$ and thus neither its range,
  so again we may assume that $\theta$ is as in the conclusion of
  \fref{thm:EmbedFactor}.
  It is also not difficult to see that
  $\essrng p = \essrng \hat\theta p \setminus \{0\} \subseteq \essrng q$.
  If the band generated by $\theta(N)$ in $N'$ is all of $N$ then
  $\hat\theta X = Y$ and $q = \hat\theta p$.
\end{proof}

In the case where $\theta$ is an isomorphism this has already been proved
by Poitevin \cite[Proposition~3.4.4]{Poitevin:PhD}.

\section{Model theory of Nakano spaces}
\label{sec:NakanoModelTheory}
\subsection{The model theoretic setting}

We assume familiarity with the general setting of continuous
first order logic, as exposed in \cite{BenYaacov-Usvyatsov:CFO} or
\cite{BenYaacov-Berenstein-Henson-Usvyatsov:NewtonMS}.
Since continuous logic only allows bounded metric structures we cannot
treat Banach spaces directly.
The two standard solutions for this are
either to consider a Banach space as a multi-sorted structure, with a
sort for $\bar B(0,n)$ (the closed ball of radius $n$) for each $n$,
or to restrict our consideration to the first of these sorts, i.e.,
the closed unit ball.
(There exists also a third solution which we shall not consider here,
namely to treat the entire Banach lattice as an
unbounded metric structure, see \cite{BenYaacov:Perturbations}.)
While Poitevin chose to use the former we consider the latter to be
preferable, so a few words regarding the difference in approaches
is in order.

The unit ball of a Banach space is, first of all, a complete convex
space, i.e., a complete metric space equipped with a
convex combination operation from an ambient Banach space.
Such structures were characterised by
Machado \cite{Machado:Convex} in a language containing all convex
combinations, and this characterisation can be expressed in continuous
logic.
There are advantages to a minimalistic language, though, so we prefer
to work in a language consisting of a single function symbol
$\half[x+y]$.
Convex combinations with coefficients of the form $\frac{k}{2^n}$ can
be obtained as more complex terms in this language, and arbitrary
convex combinations with real coefficients are obtained as limits (as
our structures are by definition complete), so this language is quite
sufficient.
While it follows from Machado's work that an axiomatisation of
unit balls of Banach spaces exists in this language, it seems
preferable to put an explicit axiomatisation of this kind on record
along with a complete (outline of a) proof.

Let $T_{cvx}$ consist of the following axioms:
\begin{align*}
  \tag{ID} & \forall x \, \left[
    \half[x+x] = x
  \right],
  &&\text{i.e.,}
  && \sup_x \left[ d\left( \half[x+x], x \right) \right ] = 0, \\
  \tag{PRM} & \forall xyzt\, \left[
    \half\left( \half[x+y] + \half[z+t] \right) =
    \half\left( \half[z+x] + \half[t+y] \right)
  \right], &&&&\text{etc.} \\
  \tag{HOM} & \forall xyz \, \left[
    d\left( \half[x+z], \half[y+z] \right) =
    \half[d(x,y)]
  \right].
\end{align*}

we shall usually be interested in subsets of Banach spaces which are
not only convex, but also contain zero and are symmetric around it
(i.e., $-x$ exists for all $x$).
The unit ball is such a space, but is not the only interesting one
(another one is the unit
ball of a von Neumann algebra with a normalised finite trace $\tau$:
it is a \emph{proper} subset of the unit ball of the Hilbert spaces
with inner product $\langle x,y\rangle = \tau(x^*y)$).
The natural language for such \emph{symmetric} convex spaces is
\begin{gather*}
  \cL_{Bs} = \{0,-,\half[x+y],\|\cdot\|\}.
\end{gather*}
we shall use $\half[x-y]$ as shorthand for $\half[x+(-y)]$.
Since we wish to admit the unit ball of a Banach space as a structure
in this language we shall interpret the distinguished distance symbol
as half the usual distance $d(x,y) = \|\half[x-y]\|$, noticing the
latter is an atomic formula.
We define $T_{sc}$ (for symmetric convex) as $T_{cvx}$ along with:
\begin{align*}
  \tag{SYM}
  & \forall x \, \left[
    \half[x-x] = 0
  \right] \\
  \tag{NORM}
  & \forall x \, \left[
    d(x,0) = \half \| x \|
  \right]
\end{align*}

Finally, we define $T_{Bs}$, the theory of (unit balls of) Banach
spaces as $T_{sc}$ along with
\begin{align*}
  \tag{FULL}
  & \forall x \exists y \, \left[
    \|x\| \geq \half \text{ or } \half[y] = x
  \right]&& \text{i.e.,} &&
  \sup_x \inf_y \, \left[
    \big( \half \dotminus \|x\| \big) \wedge d(\half[y+0],x)
  \right] = 0
\end{align*}

\begin{thm}
  \begin{enumerate}
  \item The models of $T_{cvx}$ are precisely complete convex subsets
    of diameter $\leq 1$ of Banach spaces.
  \item The models of $T_{sc}$ are
    precisely complete convex subsets of unit balls of Banach spaces
    which are symmetric around zero.
  \item The models of $T_{Bs}$ are
    precisely closed unit balls of Banach spaces.
  \end{enumerate}
\end{thm}
\begin{proof}
  For each of the assertions it is clear that all the said structures
  are models, so we prove the converse.
  we shall start by examining the case of $T_{cvx}$, reducing it to
  that of $T_{sc}$.

  From the axioms we can deduce commutativity and a variant of the
  triangle inequality:
  \begin{gather*}
    \tag{COMM}
    \half[x+y]
    = \half\left( \half[x+x] + \half[y+y] \right)
    = \half\left( \half[y+x] + \half[y+x] \right)
    = \half[y+x] \\
    \label{eq:tri2} \tag{TRI2}
    d\left( \half[x+y], \half[z+w] \right)
    \leq
    d\left( \half[x+y], \half[z+y] \right) +
    d\left( \half[z+y], \half[z+w] \right)
    = \half[d(x,z)+d(y,w)]
  \end{gather*}

  Now let $C \models T_{cvx}$.
  Let $C-C$ be the set of all formal differences $x-y$ for $x,y \in C$,
  and define $d_0(x-x',y-y') = d(\half[x+y'],\half[y+x'])$.
  This is a pseudo-metric.
  Indeed, symmetry and reflexivity are clear, and for transitivity one
  checks:
  \begin{align*}
    d(\half[x+z'],\half[z+x'])
    & = 2d\left( \half\left( \half[x+z'] + \half[y+y'] \right),
      \half\left( \half[z+x'] + \half[y+y'] \right) \right) \\
    & = 2d\left( \half\left( \half[x+y'] + \half[y+z'] \right),
      \half\left( \half[z+y'] + \half[y+x'] \right) \right) \\
    & \leq^{\fref{eq:tri2}} d\left( \half[x+y'], \half[y+x'] \right) +
    d\left( \half[y+z'], \half[z+y'] \right)
  \end{align*}
  Thus $d_0(x-y,z-t) = 0$ defines an equivalence $\sim$ relation on
  $C-C$, and $d_0$ induces a metric on
  $C_- = (C-C)/{\sim} = \{[x-y]\colon x,y \in C\}$.
  It is straightforward to verify that
  $\half[{[x-y]+[z-t]}] = \left[ \half[x+z]-\half[y+t] \right]$,
  $0 = [x-x]$ and $-[x-y] = [y-x]$ are
  well defined and render $C_-$ a model of $T_{sc}$.
  Finally, if $x_0 \in C$ is any fixed element then
  $x \mapsto [x-x_0]$ is an embedding of $C$ in $C_-$ which respects convex
  combination and shrinks distances by a factor of $2$.
  It follows that if we prove that $C_-$ embeds in a Banach space, so
  does $C$.
  We thus reduced the first assertion to the second.

  We now work modulo $T_{sc}$.
  First, observe that
  $d(x,y) = 2d(\half[x-y],\half[y-y]) = 2d(\half[x-y],0)
  = \|\half[x-y]\|$.
  Thus the relation between the distance and the norm is as expected.

  A similar reasoning shows that $\half[x+y] = 0$ implies
  $d(y,-x) = 2d(\half[x+y],\half[x-x]) = 0$, so
  $y = -x$.
  It follows that $-(-x) = x$ and that
  $-\half[x+y] = \half[-x-y]$
  (since $\half\left( \half[x+y] + \half[-x-y] \right)
  = \half\left( \half[x-x] + \half[y-y] \right) = \half[0+0] = 0$).

  Fix a model $S \models T_{sc}$.
  For $x \in S$,
  let us define $\half x = \half[x+0]$, and by induction we can
  further define $2^{-n}x$ for all $n$.
  If there is $y$ such that $x = \half y$ then $y$ is unique (indeed,
  if $z$ were another such element then
  $0 = d(x,x) = \half d(y,z)$ so $y = z$), and we may unambiguously
  write $y = 2x$.
  If $2\half[x+y]$ exists we write it as $x+y$.
  It follows from the definition that $x+0 = x$ and $x+(-x) = 0$.
  By definition we have $\half (x+y) = \half[x+y]$ (provided that
  $x+y$ exists), and applying the permutation axiom we get
  $\half x+ \half y = \half[x+y]$, from which it follows that
  $\half(-x) = -\half x$ and $\half (x+y) = \half x + \half y$
  (provided $x+y$ exists).

  From the commutativity of $\half[x+y]$ it follows that
  $x+y = y+x$, by which we mean that one exists if and only if the
  other does, in which case they are equal.
  Similarly, by the permutation axioms, if $x+y$ and $y+z$ exist
  then $(x+y)+z = x+(y+z)$.
  This means we can write something like $\sum_{i<k} x_i$ unambiguously,
  without having to specify either parentheses or order, as long as we
  know that for every subset $w \subseteq k$ the partial sum
  $\sum_{i\in w} x_i$ exists in some order and with some organisation of the
  parentheses.
  In particular, this means that
  $\sum_{i<m} k_i2^{-n_i}x_i$ always makes sense for $n_i \in \bN$, $k_i \in \bZ$
  satisfying
  $\sum 2^{-n_i}|k_i| \leq 1$, and that sums and differences of such expressions
  behave as expected
  (in particular: $2^{-n-1}x + 2^{-n-1}x = 2^{-n}x$).
  It follows that $k2^{-n}(\ell2^mx) = (k\ell)2^{-n-m}x$.

  It follows directly from the axioms that
  $\|\half x\| = \half d(\half x,0) = \half\cdot\half d(x,0) = \half \|x\|$.
  We obtain $\|x\| = 2d(0,x) = 2\left\| \half[0-x] \right\| = \|-x\|$,
  and if $x+y$ exists then
  $\|x+y\| = 2d(x+y,0) \leq 2d(x+y,y) + 2d(y,0)
  = 2\left\|\half[(x+y)-y]\right\| + 2\left\|\half[y-0]\right\|
  = \|x\| + \|y\|$.
  By induction on $n$ one proves first that $\|2^{-n}x\| = 2^{-n}\|x\|$,
  and then that for all $0 \leq k \leq 2^n$:
  $\|k2^{-n}x\| = k2^{-n}\|x\|$.
  It follows that $\|\sum_{i<m} k_i2^{-n}x_i\| \leq 2^{-n}\sum |k_i|$.

  Thus for every
  $\alpha \in [-1,1]$ we can define $\alpha x$ as a limit of $k_n2^{-n}x$.
  We obtain that $\sum \alpha_ix_i$ always makes sense if
  $\sum |\alpha_i| \leq 1$,
  $\alpha(\beta x) = (\alpha\beta)x$,
  $(\alpha+\beta)x = \alpha x+\beta x$
  (provided that $|\alpha+\beta| \leq 1$),
  $\alpha(x+y) = \alpha x+\alpha y$
  (provided that $x+y$ exists),
  and $\|\alpha x\| = |\alpha| \|x\|$.
  We also have $d(\alpha x,\alpha y) = \left\| \half[\alpha x-\alpha y] \right\|
  = |\alpha| \left\| \half[x-y] \right\| = |\alpha|d(x,y)$, so in particular
  $\alpha x = \alpha y \Longrightarrow x=y$ for $|\alpha| \neq0$.

  We can now define $B_0 = \bR^{>0} \times S$,
  and define $(\alpha,x) \sim (\beta,y)$
  if $\frac{\alpha}{\alpha+\beta}x = \frac{\beta}{\alpha+\beta}y$.
  It is straightforward to verify using results from the previous
  paragraph that $\sim$ is an equivalence relation, and that the
  following operations are well defined on $B = B_0/{\sim}$ and render it
  a normed vector space over $\bR$:
  \begin{align*}
    & \beta[\alpha,x] =
    \begin{cases}
      [\alpha\beta,x] & \beta > 0 \\
      [-\alpha\beta,-x] & \beta < 0 \\
      [1,0] & \beta = 0
    \end{cases}\\
    & [\alpha,x] + [\beta,y] =
    \left[ \alpha+\beta,\frac{\alpha}{\alpha+\beta}x
      + \frac{\beta}{\alpha+\beta}y \right] \\
    & \|[\alpha,x]\| = \alpha \|x\|.
  \end{align*}
  Our structure $S$ embeds in the unit ball of $B$ via
  $x \mapsto [1,x]$.

  The last assertion now follows immediately.
\end{proof}

When dealing with models of $T_{Bs}$ 
we allow ourselves to omit the halving operation when no
ambiguity may arise.
Thus, for example, we write $x+y+z = t+w$ instead of
$\half \left( \half[x]+\half[y+z] \right) = \half \half[t+w]$, and so on.

We shall now extend this to Banach lattices.
We recall a few definitions from \cite{MeyerNieberg:BanachLattices}:
\begin{dfn}
  \label{dfn:BanLat}
  \begin{enumerate}
  \item An \emph{ordered vector space} $(E,\leq)$ is a vector space
    $E$ equipped with a partial ordering $\leq$, over an
    ordered field $(k,\leq)$,
    satisfying
    \begin{gather*}
      v \leq u
      \quad \Longleftrightarrow \quad v+w \leq u + w
      \quad \Longleftrightarrow \quad \alpha v \leq \alpha u
    \end{gather*}
     for all $v,u,w \in E$ and $\alpha \in k^{>0}$.
  \item An ordered vector space is a \emph{vector lattice} (or a
    \emph{Riesz space}) if
    it is a lattice, i.e., if every two $v,u \in E$ admit a
    least upper bound (or \emph{join}) $v\vee u$ and a greatest lower
    bound (or \emph{meet}) $v\wedge u$.
    In this case we write $|v| = v\vee(-v)$,
    $v^+ = v\vee0$, $v^- = (-v)\vee0$.
  \item \label{item:NormLat}
    A \emph{normed vector lattice} is
    a vector lattice over $\bR$, equipped with
    a norm satisfying $|v| \leq |u| \Longrightarrow \|v\| \leq \|u\|$.
  \item A \emph{Banach lattice} is a complete normed vector lattice.
  \end{enumerate}
\end{dfn}

We shall consider (unit balls of) Banach lattices
in a language augmented with a $1$-Lipschitz function symbol:
\begin{align*}
  \cL_{Bl} & = \cL_{Bs} \cup \{|\cdot|\}.
\end{align*}
Using the function symbol $|\cdot|$ we may define other common
expressions which have the intended interpretations in Banach lattices:
\begin{gather*}
  x^+ = \half[|x|+x], \qquad x^- = \half[|x|-x], \\
  \half[x\vee y]
  = \half \left( \half[x+y] + \left| \half[x-y] \right| \right), \qquad
  \half[x \wedge y]
  = \half \left( \half[x+y] - \left| \half[x-y] \right| \right).
\end{gather*}
On the other hand we cannot expect to define $x \vee y$ or
$x \wedge y$ without halving since the unit ball of a Banach lattice
need not be closed under these operations.

We define $T_{Bl}$ to consist of $T_{Bs}$ along with the following
axioms.
We shall follow the convention (which will be justified later)
that $x \geq 0$ is shorthand for $x = |x|$.
\begin{align*}
  \tag{BL1}
  & | \alpha x | = |\alpha| |x| && \alpha \in [-1,1] \text{ dyadic} \\
  \tag{BL2}
  & \half[|x|+|y|] \geq 0 \\
  \tag{BL3}
  & \|x\| = \bigl\| |x| \bigr\| \leq \bigl\| |x| + |y| \bigr\| \\
  \tag{BL4}
  & |x^+| = x^+ \\
  \tag{BL5}
  & \half[z] - \half[x \vee y] + (\half[z-x])^- + (\half[z-y])^- \geq 0
\end{align*}
(Some halving is omitted from BL3,5.)

\begin{thm}
  If $(E,\leq)$ is a Banach lattice then the unit ball of $E$ is
  closed under the absolute value operation $|\cdot|$ and as a
  $\cL_{Bl}$-structure is a model of $T_{Bl}$.
  Conversely, every model of $T_{Bl}$ is the unit ball of a Banach
  lattice, where the absolute value operation is extended to the
  entire Banach space by
  $|x| = \|x\| \left| \frac{x}{\|x\|} \right|$
  and the order is recovered by
  $x \leq y \Longleftrightarrow x-y = |x-y|$.
\end{thm}
\begin{proof}
  The first statement is immediate so we only prove the converse.
  If $(E_1,|\cdot|) \models T_{Bl}$ then $E$ is the unit ball of a
  Banach space $E$.
  By BL1 we may extend the absolute value operation to all
  of $E$ as in the statement of the Theorem and have
  $|\alpha x| = |\alpha| |x|$ for all $\alpha \in \bR$, $x \in E$.
  By BL2 $\bigl| |x| + |y| \bigr| = |x| + |y|$.

  Define a relation $\leq$ on $E$ as in the statement.
  Clearly $x-x = 0 = |0|$ whereby $x \leq x$.
  If $x \leq y \leq x$
  then $0 = |x-y| + |y-x|$ and by BL3
  $\|x-y\| \leq 0$, i.e., $x=y$.
  If $x \leq y \leq z$ then
  $z-x = |y-x| + |z-y|$ whereby
  $z-x = |z-x|$, i.e.,
  $x \leq z$.
  Thus $\leq$ is an ordering and it is now clear that it renders
  $E$ an ordered vector space.
  In particular, $x \geq 0 \Longleftrightarrow x = |x|$,
  justifying our notation.

  Define
  $x \vee y = \half[x+y] + \left| \half[x-y] \right|$,
  $x \wedge y = \half[x+y] + \left| \half[x-y] \right|$.
  Then
  $x\vee y - x = \left( \half[y-x] \right)^+$,
  and by BL4
  $x \vee y \geq x$.
  The inequalities $x\vee y \geq y$
  and $x,y \geq x\wedge y$ are proved similarly.
  Assume now that
  $z \geq x,y$.
  Then $(\half[z-x])^- = \half (z-x)^-$ and similarly for $z-y$,
  and by BL5 $z \geq x \vee y$.
  Thus $x \vee y$ is the join of $x$ and $y$.
  It is not difficult to check that
  $x\wedge y = -( (-x) \vee (-y) )$ is the meet of $x$ and $y$,
  so $(E,\leq)$ is a Riesz space.
  Immediate calculations also reveal that
  $|x| = x \vee (-x)$, $x^+ = x\vee 0$, $x^- = (-x)\vee 0$.

  Finally, if $|x| \leq |y|$,
  applying BL3 to $|x|$ and $|y|-|x|$
  we obtain
  $\| x \| = \bigl\| |x| \bigr\| \leq \bigl\| |y| \bigr\|
  = \|y\|$.
  This completes the proof.
\end{proof}

There is nothing sacred in our choice of language, and some may prefer
to name the operations $\half[x \vee y]$, $\half[x \wedge y]$ instead
of the absolute value, thus working in
$\cL_{Bl}' = \cL_{Bs} \cup \{\half[x \vee y],\half[x \wedge y]\}$.
We have seen that $\half[x \vee y]$, $\half[x \wedge y]$
can be written as terms using $|\cdot|$,
so every atomic $\cL_{Bl}'$-formula can be translated to an atomic
$\cL_{Bl}$-formula.
The converse is not true, but we may still write
$\half[ |x| ] = \half[x\vee(-x)]$.
An easy induction on the complexity of terms yields that every atomic
$\cL_{Bl}$-formula can be expressed as an atomic $\cL_{Bl}'$-formula
up to a multiplicative factor of the form $2^k$,
and therefore as a quantifier-free $\cL_{Bl}'$-formula.
We may therefore say that the two languages are quantifier-free
interpretable in one another.
By \fref{thm:BiInterpretableClasses},
model theoretic properties such as
axiomatisability, quantifier elimination, model
completeness, and so on, transfer from any class of
Banach lattices viewed as structures in one language to the same class
viewed as structures in the other.
One could also formalise Banach lattices by naming the operation
$x^+$ (or $x^-$), and the same argument would hold.

Since we are dealing specifically with Nakano spaces, we may consider
them in the language
$\cL_{Bl}^\Theta = \cL_{Bl} \cup \{\Theta\}$
where $\Theta$ will interpret
the modular functional.
However, there is a small caveat here: the modular functional
$\Theta_{p(\cdot)}$ is indeed uniformly continuous on the unit ball of
$L_{p(\cdot)}(X,\fB,\mu)$, but its precise uniform
continuity modulus depends on the essential bound of the
exponent function $p$.

\begin{conv}
  We fix here, once and for all, a \emph{uniform} bound
  $1 \leq r < \infty$ on $p$.
  Thus all Nakano spaces considered henceforth will be of the form
  $L_{p(\cdot)}(X,\fB,\mu)$ where $p\colon X \to [1,r]$.
\end{conv}

Let $K \subseteq [1,r]$ be compact.
we shall consider the following classes of structures:
\begin{align*}
  \cN^\Theta_K
  & = \{\text{$\cL_{Bl}^\Theta$-structures isomorphic to some 
    $(L_{p(\cdot)}(X,\fB,\mu),\Theta_{p(\cdot)})$ with $\essrng p = K$}\}, \\
  \cN_K
  & = \{N\rest_{\cL_{Bl}}\colon N \in \cN^\Theta_K\} \\
  & = \{\text{$\cL_{Bl}$-structures isomorphic to some 
    $L_{p(\cdot)}(X,\fB,\mu)$ with $\essrng p = K$}\}, \\
  \cN^\Theta_{\subseteq K}
  & = \bigcup \{ \cN^\Theta_{K'}\colon \emptyset \neq K' \subseteq K \text{ compact}\}, \\
  & = \{\text{$\cL_{Bl}^\Theta$-structures isomorphic to some 
    $(L_{p(\cdot)}(X,\fB,\mu),\Theta_{p(\cdot)})$ with $\essrng p \subseteq K$}\}, \\
  \cN_{\subseteq K}
  & =  \bigcup \{ \cN_{K'}\colon \emptyset \neq K' \subseteq K \text{ compact}\}
  = \{N\rest_{\cL_{Bl}}\colon N \in \cN^\Theta_{\subseteq K}\} \\
  & = \{\text{$\cL_{Bl}$-structures isomorphic to some 
    $L_{p(\cdot)}(X,\fB,\mu)$ with $\essrng p \subseteq K$}\}.
\end{align*}
(Of course, strictly speaking, these are the classes of the unit balls
rather than of entire spaces.)

Given the uniform bound we fixed before, the largest classes we may
consider are $\cN_{\subseteq[1,r]}$ and $\cN^\Theta_{\subseteq[1,r]}$,
respectively.

\begin{fct}
  \label{fct:NakanoElemClass}
  Each of the classes
  $\cN^\Theta_K$, $\cN_K$, $\cN^\Theta_{\subseteq K}$
  and $\cN_{\subseteq K}$
  is elementary.
\end{fct}
\begin{proof}
  This is just \cite[Proposition~3.8.2]{Poitevin:PhD}.
  While the case of $\cN^\Theta_{\subseteq K}$ is not mentioned there explicitly
  all the ingredients are there (in particular, as each class
  of the form $\cN^\Theta_K$ is closed under ultraroots, so are classes of
  the form $\cN^\Theta_{\subseteq K}$).
\end{proof}

We may impose additional requirement, such as the dimension being
greater than $1$, or the lattice (equivalently, the underlying measure
space) being atomless.
These are first order conditions as well.
For the first one we would like to say that there are functions $x$
and $y$ such that $\|x\| = \|y\| = 1$ and $|x|\wedge|y| = 0$,
i.e.:
\begin{gather*}
  \inf_{x,y}
  \neg\|x\|
  \vee
  \neg\|y\|
  \vee
  \bigl\| |x+y| - |x-y| \bigr\| = 0.
\end{gather*}
Similarly, atomlessness is expressible by:
\begin{gather*}
  \sup_x \inf_y
  \bigl|  \| y \| - \half[\| x \|] \bigr|
  \vee
  \bigl\| |x| - | x - 2y | \bigr\| = 0.
\end{gather*}

The classes of Nakano spaces of dimension at least $2$ will be denoted
$2\cN_K$, $2\cN^\Theta_K$, etc.
The classes of atomless Nakano spaces will be denoted
$\cA\cN_K$, $\cA\cN^\Theta_K$, etc.

\begin{fct}
  \label{fct:FixedEssRng}
  Assume $L_{p(\cdot)}(X,\fB,\mu) \in 2\cN_K$ ($\in 2\cN_{\subseteq K}$).
  Then $\essrng p = K$ ($\subseteq K$).
\end{fct}
\begin{proof}
  This is a consequence of \cite[Proposition~3.4.4]{Poitevin:PhD},
  which can be also obtained as a special case of
  \fref{cor:RigidEssRng}.
\end{proof}

\begin{fct}
  The theory $\Th(\cA\cN^\Theta_K)$ eliminates quantifiers.
  It follows that it is complete, as is $\Th(\cA\cN_K)$.
\end{fct}
\begin{proof}
  \cite[Theorem~3.9.4]{Poitevin:PhD}.
\end{proof}

\begin{fct}
  Let $K \subseteq (1,\infty)$ be compact (so $\min K > 1$).
  Then the theory $\Th(\cA\cN_K)$ is stable.
\end{fct}
\begin{proof}
  \cite[Theorem~3.10.9]{Poitevin:PhD}.
\end{proof}

In fact, we are cheating here a little, as Poitevin proved his results
in a somewhat different language.
He follows the approach described in the paragraphs following
\cite[Example~4.5]{BenYaacov-Usvyatsov:CFO}, viewing a Banach space $N$ as
multi-sorted structure consisting of a sort
$N_m = \bar B(0,m)$ for each $0 < m < \omega$.
The corresponding language for Banach lattices, which we may denote
here by $\cL_{Bl,\omega}$, consists of the obvious embedding mappings
between sorts, plus multiplication by (say, rational) scalars and the
binary operations $+$, $\wedge$ and $\vee$ going from sorts or pairs of sorts
to an appropriate target sort (e.g., $+\colon N_m\times N_k \to N_{m+k}$, or
$\half x\colon N_2 \to N_1$).
The predicate symbols norm and distance can have values greater than
one, but they are still bounded on each sort and thus still fit in the
framework of continuous logic.
Similarly, one can define
$\cL_{Bl,\omega}^\Theta$ as
$\cL_{Bl,\omega}$ along with a
predicate symbol $\Theta$ on each sort, and again in every Nakano space
$\Theta$ is uniformly continuous and bounded on each sort.

It will be convenient to notice that even in this approach,
multiple sorts are not required.
Since all the sorts $N_m$ of a Banach space stand in a natural
bijection with the unit ball sort $N_1$ via dilation
$x \mapsto \frac{x}{m}$, we may interpret the entire language
$\cL_{Bl,\omega}$ on the single sort $N_1$.
Thus, for example, instead of $+\colon N_m \times N_k \to N_{m+k}$
we would have $+_{m,k}\colon N_1 \times N_1 \to N_1$ sending
$(\frac{x}{m},\frac{y}{k}) \mapsto \frac{x+y}{m+k}$.
Viewing $N_1$ as itself, rather than as a scaled copy of $N_m$,
$N_k$ or $N_{m+k}$, obtain the convex combination operation
$x +_{m,k} y = \frac{mx+ky}{m+k}$.
In particular, $x +_{1,1} y = \half[x+y]$.

Viewed in this way, $\cL_{Bl}$ ($\cL_{Bl}^\Theta$)
is a sub-language of $\cL_{Bl,\omega}$ ($\cL_{Bl,\omega}^\Theta$).
It is also fairly immediate to check that every atomic
$\cL_{Bl,\omega}$-formula agrees (in any Banach lattice)
with a quantifier-free $\cL_{Bl}$-formula.
Thus $\cL_{Bl}$ and $\cL_{Bl,\omega}$ are quantifier-free
bi-interpretable, in the sense of \fref{apx:Interpretations},
on the class of Banach lattices.
By \fref{thm:BiInterpretableClasses}, model theoretic properties such
as
elementarity, model completeness, quantifier elimination, and so
on, transfer between classes of Banach lattices formalised
in $\cL_{Bl}$ and in $\cL_{Bl,\omega}$.
(The reader may worry that in the single sorted versions of
$\cL_{Bl}$ and $\cL_{Bl}^\Theta$ we may construct terms and
formulae which do not come from the multi-sorted version due to sort
discrepancy, for example the term
$x +_{m,k} (y +_{\ell,t} z)$ where $k \neq \ell+t$.
This term, however, agrees with the ``legitimate'' term
$x +_{m(\ell+t),k(\ell+t)} (y +_{k\ell,kt} z)$ in every Banach
lattice.
In this fashion we can translate every term or quantifier-free formula
of $\cL_{Bl,\omega}$ to one which would make sense in the multi-sorted
version, so this is not a true problem.)

Let us now consider the case of
$\cL_{Bl}^\Theta \subseteq \cL_{Bl,\omega}^\Theta$.
The language $\cL_{Bl,\omega}^\Theta$
contains for every $m$ a predicate symbol
$\Theta_m\colon N_1 \to \bR^+$,
$\Theta_m(x) = \Theta(mx)$ (the range of $\Theta_m$ is
bounded and the bound depends only on $r$ and $m$),
while $\cL_{Bl}^\Theta$ only contains the first one of those,
$\Theta = \Theta_1$.
Unlike the predicates for norm and distance on $N_m$ which are
homogeneous and can therefore be
recovered from their counterparts on $N_1$ by simple dilation,
in order to recover $\Theta_m$ from $\Theta_1$ a little more work is
required.
Our argument here is very close to the proof of
\cite[Lemma~3.4.1]{Poitevin:PhD}.
Let us first recall a version of the Stone-Weierstrass Density
Theorem:
\begin{fct}
  Let $X$ be a compact Hausdorff space and let $\sA \subseteq \cC(X,\bR)$ be a
  sub-algebra which separates points and vanishes nowhere
  (i.e., for each $x \in X$ there is $f \in \sA$ such that
  $f(x) \neq 0$).
  Then $\sA$ is dense in $\cC(X,\bR)$.
\end{fct}

\begin{lem}
  For every $0 < m \in \bN$
  there exists a quantifier-free $\cL_{Bl}^\Theta$-definable predicate
  $\varphi_m(x)$ which coincides with
  $v \mapsto \Theta(mv)$ on the unit ball of every Nakano space
  $N = L_{p(\cdot)}(X,\fB,\mu)$ (with $\essrng(p) \subseteq [1,r]$).
\end{lem}
\begin{proof}
  Let $\sA\subseteq \cC([1,r],\bR)$
  consist of all functions of the form
  $f(x) = \sum_{i<n} a_k2^{-kx}$, where $n \in \bN$
  and $a_k \in \bR$.
  Then $\sA$ satisfies the assumptions of the
  Stone-Weierstrass Density Theorem cited above,
  and is therefore dense in $\cC([0,1],\bR)$.

  Let us fix $\varepsilon > 0$.
  By the previous paragraph there is a function of the form
  $f(x) = \sum_{k<n} a_k2^{-kx} \in \sA$
  which is $\varepsilon$-close to $g(x) = m^x$ on $[1,r]$.
  Then $\varphi_{m,\varepsilon}(v) = \sum_{k<n} a_k\Theta(2^{-k}v)$
  is a quantifier-free definable predicate in $\cL_{Bl}^\Theta$.

  Now assume that $v \in N = L_{p(\cdot)}(X,\fB,\mu)$,
  $\|v\| \leq 1$.
  Passing to $|v|$ we may assume that $v \geq 0$ and up
  to a density change we may assume that $v = \chi_A$
  for some $A \in \fB$.
  Then $\|v\| \leq 1$ implies that $\mu(A) \leq 1$.
  Consider the restriction $p\rest_A \colon A \to[1,r]$,
  and let $\nu$ be the image measure of $\mu\rest_A$
  under this mapping.
  For every $\alpha > 0$ we have
  \begin{gather*}
    \Theta(\alpha v)
    = \int_A \alpha^{p(x)}\,d\mu
    = \int_{[1,r]} \alpha^x\,d\nu,
  \end{gather*}
  whereby
  \begin{align*}
    \bigl| \Theta(mv) - \varphi_{m,\varepsilon}(v) \bigr|
    & = \left|
      \int_{[1,r]} m^x\,d\nu  - \sum_{k<n} \int_{[1,r]} 2^{-kx}\,d\nu
    \right| \\
    & \leq \int_{[1,r]} |f(x) - g(x)|\, d\nu
    \leq \varepsilon \mu(A) \leq \varepsilon.
  \end{align*}
  Since this can be done for every $\varepsilon > 0$ the statement is proved.
\end{proof}

Thus $\cL_{Bl}^\Theta$ and $\cL_{Bl,\omega}^\Theta$ are also
quantifier-free bi-interpretable for Nakano spaces,
so axiomatisability, quantifier elimination and so on
transfer between the two formalisms.
This also means that once we show that the modular functional
of a Nakano space is $\cL_{Bl}$-definable in the unit ball
(e.g., \fref{thm:DefinableModular}), it follows that
it is $\cL_{Bl}$-definable on the $m$-ball for every $m$.

\section{Definability of the modular functional}
\label{sec:DefinableModular}

This section contains the main model theoretic results of this paper.
We start with the definability result.

\begin{thm}
  \label{thm:DefinableModular}
  The modular functional $\Theta$ is uniformly $\cL_{Bl}$-definable
  in $2\cN^\Theta_{\subseteq[1,r]}$.
  Moreover, it is both uniformly $\inf$-definable and $\sup$-definable
  and can be used
  to axiomatise $2\cN^\Theta_{\subseteq[1,r]}$
  modulo the axioms for $2\cN_{\subseteq[1,r]}$.

  More precisely:
  \begin{enumerate}
  \item There exists a $\cL_{Bl}$-definable predicate
    $\varphi_\Theta(x)$ such that
    $(N,\Theta) \models \Theta(x) = \varphi_\Theta(x)$
    for all $(N,\Theta) \in 2\cN^\Theta_{\subseteq[1,r]}$.
  \item There quantifier-free $\cL_{Bl}$-formulae
    $\psi_n(x,\bar y_n)$ and $\chi_n(x,\bar z_n)$ such that in all
    Nakano spaces of dimension at least two:
    $$\Theta(x) = \varphi_\Theta(x)
    = \lim_{n\to\infty} {\inf}_{\bar y_n} \psi_n(x,\bar y_n)
    = \lim_{n\to\infty} {\sup}_{\bar z_n} \chi_n(x,\bar z_n),$$
    each of the limits converging uniformly and at a uniform rate.
  \item The theory $\Th(2\cN^\Theta_{\subseteq[1,r]})$ is equivalent to
    $\Th(2\cN_{\subseteq[1,r]}) \cup \{\Theta(x) = \varphi_\Theta(x)\}$.
  \end{enumerate}
\end{thm}
\begin{proof}
  By \fref{cor:RigidModular} every $N \in 2\cN_{\subseteq[1,r]}$ admits at most
  one expansion to $(N,\Theta) \in 2\cN^\Theta_{\subseteq[1,r]}$.
  As these are elementary classes, one can apply
  \fref{thm:Beth} (Beth's theorem for continuous logic) in order to
  obtain $\varphi_\Theta$.

  Using \fref{cor:RigidModular} again we see that $\varphi_\Theta$ is
  constant in $2N_{\subseteq[1,r]}$ (see \fref{dfn:IncDecFormula}).
  By \fref{thm:MonForm} it is both $\inf$-definable and
  $\sup$-definable there.

  The last item is immediate.
\end{proof}

\begin{cor}
  \label{cor:DefModK}
  For a fixed compact $K \subseteq [1,r]$, the modular functional is
  uniformly $\cL_{Bl}$-definable in $\cN^\Theta_K$.

  In particular the modular functional is $\cL_{Bl}$-definable in
  every Nakano Banach lattice.
\end{cor}
\begin{proof}
  If $K = \{p_0\}$ is a single point, we have $\Theta(f) = \|f\|^{p_0}$.
  Otherwise
  $\cN^\Theta_K = 2\cN^\Theta_K
  \subseteq 2\cN^\Theta_{\subseteq[1,r]}$
  and we can apply
  \fref{thm:DefinableModular}.
\end{proof}

We have shown that naming the modular functional does not add
structure.
Still, in the case of an atomless Nakano space naming $\Theta$ does give
something, namely quantifier elimination.
It is clear that without $\Theta$ quantifier elimination would be
impossible: the complete $\cL_{Bl}^\Theta$-type of a function contains,
among other information, the essential range of $p$ on its support,
and there is no way of recovering this information from the
quantifier-free $\cL_{Bl}$-type of a single positive function, as it
is determined by its norm alone.

A next-best would be to obtain model completeness.
Indeed, all the work for obtaining it is already done.

\begin{thm}
  For every compact $K \subseteq [1,r]$ the (theory of the) class $\cA\cN_K$
  is model complete.
\end{thm}
\begin{proof}
  Follows from \fref{cor:RigidModular} and the quantifier
  elimination in $\cA\cN^\Theta_K$.
\end{proof}

The next and last result of this section is quite quick and
straightforward to prove for a person who is quite familiar with the
notion of a \emph{measure algebra} and understands that
\fref{thm:EmbedFactor} is actually a result about measure algebras
rather than about measure spaces.
Having intentionally avoided all mention of measure algebras so
far, a longer approach is required, presenting this somewhat
different point of view.
We introduce measure algebras in a very sketchy fashion,
as an abstract version of a (strictly localisable) measure space.
For a comprehensive treatment we refer the reader to
\cite{Fremlin:MeasureTheoryVol3}.

Let $(X,\fB,\mu)$ be a measure space.
Let $\fB_f \subseteq \fB$ be the lattice of finite measure sets.
As an algebraic structure, $(\fB_f,\cup,\cap,\setminus)$ is a
relatively complemented distributive lattice (and if it contains a
maximal element then it is a Boolean algebra).
The measure $\mu\colon \fB_f \to \bR^+$ induces a pseudo-metric
$d(x,y) = \mu(x \triangle y)$ on $\fB_f$.
The kernel of this pseudo-metric, namely the equivalence relation
$d(x,y) = 0$, is compatible with the algebraic structure,
yielding a quotient relatively complemented distributive lattice
$(\overline \fB,\cap,\cup,\setminus)$.
The measure function $\mu$ induces an additive ``measure function''
$\bar \mu\colon \overline \fB \to \bR^+$,
and $d(x,y) = \bar \mu(x \triangle y)$ is a metric on
$\overline \fB$, with respect to which the operations
$\cap,\cup,\setminus$ are $1$-Lipschitz.
Moreover, it follows from $\sigma$-additivity of the original measure
that $\overline \fB$ is a complete metric space.
For the purpose of the discussion that follows,
we call $(\overline \fB,\cup,\cap,\setminus,\bar \mu)$ the
\emph{measure algebra} associated to $(X,\fB,\mu)$.

Conversely, let $(\fC,\cup,\cap,\setminus,\nu)$ be an
abstract measure algebra, namely a relatively
complemented distributive lattice where $\nu\colon \fC\to \bR^+$ is
additive, such that in addition
$d(x,y) = \nu(x \triangle y)$ is a complete metric on $\fC$.
Assume first that $\fC$ contains a maximal element $1$,
i.e., that $\fC$ is a Boolean algebra.
Let $\widetilde \fC$ be its Stone space.
For $x \in \fC$ let $\tilde x \subseteq \widetilde \fC$
be the corresponding clopen set,
and define $\tilde \nu_0(\tilde x) = \nu(x)$.
Then Carathéodory's Extension Theorem applies and we may extend
$\tilde \nu_0$ uniquely to a regular Borel measure $\tilde \nu$
on $\widetilde \fC$.
It is now easy to check that $(\fC,\cup,\cap,\setminus,\nu)$
is the measure algebra associated to the measure space
$(\widetilde\fC,\tilde \nu)$ (equipped with the Borel
$\sigma$-algebra).
In the general case let $\{a_i\}_{i\in I} \subseteq \fC$ be a maximal
disjoint family of non-zero members.
For each $i$ let $\fC_i = \{b\cap a_i\}_{b \in \fC}$ be the
restriction of $\fC$ to $a_i$.
Restriction the other operations we obtain a
measure algebra $(\fC_i,\cup,\cap,\setminus,\nu_i)$ with a maximal
element $a_i$, so the previous argument works.
The disjoint union
$\coprod_{i\in I} (\widetilde \fC_i,\tilde \nu_i)$ is a strictly
localisable measure space, and it is not difficult to check that its
measure algebra is (canonically identified with) $\fC$.

\begin{dfn}
  Let $\alpha > 0$ be an ordinal,
  $\{N_i\}_{i < \alpha}$ an increasing chain of members of
  $2\cN_{\subseteq [1,r]}$
  (as usual, all inclusions are assumed to be isometric).

  A \emph{compatible presentation} for this sequence is a
  sequence of presentations
  $N_i \cong L_{p_i(\cdot)}(X_i,\fB_i,\mu_i)$
  such that each inclusion $N_i \subseteq N_j$
  sends characteristic functions to characteristic functions.
\end{dfn}

\begin{lem}
  Let $\alpha > 0$ be a limit ordinal,
  $\{N_i\}_{i < \alpha}$ an increasing chain of members of
  $2\cN_{\subseteq [1,r]}$.
  Let $N_i \cong L_{p_i(\cdot)}(X_i,\fB_i,\mu_i)$, $i < \alpha$,
  be a compatible presentation for this sequence.
  Let $N_\alpha = \bigcup_{i<\alpha} N_i$ in the
  sense of continuous logic, namely the metric completion of the
  set-theoretic union.
  Then $N_\alpha \in 2\cN_{\subseteq [1,r]}$ as well.

  Moreover, there exists a presentation
  $N_\alpha \cong L_{p_\alpha(\cdot)}(X_i,\fB_i,\mu_i)$
  which extends the original compatible presentation to one for the
  sequence $\{N_i\}_{i \leq \alpha}$,
  and $\essrng p_\alpha = \overline{\bigcup_i \essrng p_i}$
\end{lem}
\begin{proof}
  For $i < \alpha$ let
  $\overline \fB_i$ be the measure algebra associated to the
  measure space $(X_i,\fB_i,\mu_i)$.
  The compatibility assumption tells us precisely that
  for $i < j < \alpha$,
  the embedding $N_i \subseteq N_j$ induces an embedding
  $\overline \fB_i \subseteq \overline \fB_j$
  which respects the algebraic structure as
  well as the measure.
  We may therefore define
  $\fC = \widehat{\bigcup_{i<\alpha} \overline\fB_i}$
  (i.e., the completion of the union).
  Since the algebraic (lattice) operations as well as the measure
  function are uniformly continuous, they
  extend uniquely to $\fC$, rendering it an abstract measure
  algebra.
  By the discussion
  above we may identify it with the measure algebra
  $\overline \fB_\alpha$ of some measure space
  $(X_\alpha,\fB_\alpha,\mu_\alpha)$.

  For each $i <j  \leq \alpha$,
  the embedding $\overline \fB_i \subseteq \overline \fB_j$
  induces a partial mapping
  $L_0(X_i,\fB_i,\mu_i) \dashrightarrow L_0(X_j,\fB_j,\mu_j)$
  defined on the space of simple functions.
  By \fref{lem:SetMorphism} this extends uniquely to a total mapping
  $\hat \theta_{ij}\colon
  L_0(X_i,\fB_i,\mu_i) \to L_0(X_j,\fB_j,\mu_j)$.
  Moreover, for $i < j < k \leq \alpha$ we have
  $\hat \theta_{jk} \circ \hat \theta_{ij} = \hat \theta_{ik}$
  and
  $\hat \theta_{ij}\rest_{N_i}$
  coincides with the inclusion $N_i \subseteq N_j$.

  By \fref{thm:EmbedFactor} we have
  $\hat \theta_{ij} p_i = \chi_{\hat \theta_{ij}X_i} p_j$,
  whence
  $\hat \theta_{i\alpha} p_i \rest_{\hat \theta_{i\alpha}X_i}
  = \hat \theta_{j\alpha} p_j \rest_{\hat \theta_{i\alpha}X_i}$
  for $i < j < \alpha$.
  It is also not difficult to check
  that
  $X_\alpha = \bigcup_{i<\alpha} \hat \theta_{i\alpha}X_i$
  up to null measure
  (more precisely, that every finite measure
  $A \in \fB_\alpha$ is contained, up to arbitrarily small measure,
  by some $\hat \theta_{i\alpha}X_i$),
  so there exists a unique measurable
  $p_\alpha \colon X_\alpha \to [1,r]$ such that
  $\hat \theta_{i\alpha} p_i
  = \chi_{\hat \theta_{i\alpha}X_i} p_\alpha$ for all $i< \alpha$,
  and its essential range is as stated.
  Let
  $N_\alpha' = L_{p_\alpha(\cdot)}(X_\alpha,\fB_\alpha,\mu_\alpha)$.
  We obtain embeddings
  $\hat\theta_{i,\alpha}\rest_{N_i} \colon N_i \to N_\alpha'$.
  Moreover, every characteristic function of a finite measure set in
  $N_\alpha'$ is arbitrarily well approximated by members of the set
  union $\bigcup_i \hat\theta_{i,\alpha}(N_i)$
  (by construction of $\fC$).
  It follows that the image of the set union is dense, whence
  we get an isomorphism
  $N_\alpha
  \cong N_\alpha'
  = L_{p_\alpha(\cdot)}(X_\alpha,\fB_\alpha,\mu_\alpha)$
  which respects characteristic functions, as desired.
\end{proof}

\begin{thm}
  The (theories of the) classes
  $\cN_K$, $\cN_{\subseteq K}$, $\cN_K^\Theta$,
  $\cN_{\subseteq K}^\Theta$, and similarly with prefixes
  $2$ and $\cA$, are all inductive.
\end{thm}
\begin{proof}
  It is immediate from the previous Lemma that
  $2\cN_{\subseteq K}$ and $2\cN_{\subseteq K}^\Theta$ are inductive.
  It follows that $\cN_{\subseteq K}$
  and $\cN_{\subseteq K}^\Theta$ are inductive, since every infinite
  increasing chain in this classes has a tail in
  $2\cN_{\subseteq K}$ or in $2\cN_{\subseteq K}^\Theta$.
  Since the atomlessness axiom is inductive,
  the classes 
  $\cA\cN_{\subseteq K}$ and $\cA\cN_{\subseteq K}^\Theta$ are
  inductive.
  The same reasoning works for $K$ instead of $\subseteq K$.
  (Of course, for $\cA\cN_K$ and $\cA\cN_K^\Theta$,
  inductiveness follows directly from model completeness).
\end{proof}

\section{Perturbations of the exponent}
\label{sec:PertExp}

Intuitively, a small change to the exponent function $p$ should not
change the structure of a Nakano space by too much.
We formalise this intuitive idea, showing that small perturbations of
the exponent form indeed a perturbation system in the sense of
\cite{BenYaacov:Perturbations}.
We show that up to such perturbations, every complete theory of Nakano
spaces is $\aleph_0$-categorical and $\aleph_0$-stable.
In case $p$ is constant (i.e., $K$ is a singleton), we already know
(see, e.g., \cite{BenYaacov-Berenstein-Henson:LpBanachLattices})
that the theory is $\aleph_0$-stable and
$\aleph_0$-categorical without perturbation.
Indeed, no perturbation of $p$ is possible in this case, so it is
a special case of what we prove below.

\subsection{Preliminary computations}

We seek bounds for $1+\gamma^s$ in terms of $(1+\gamma)^s$, and for
$1-\gamma^s$ in terms of $(1-\gamma)^s$,
where $\gamma \in [0,1]$ and $s \in [1/r,r]$.
The function $\frac{1+\gamma^s}{(1+\gamma)^s}$ is well behaved, i.e., continuous
as a function of two variables, and will not cause trouble.
The function $\frac{1-\gamma^s}{(1-\gamma)^s}$
is badly behaved near $\gamma = 1$,
so we shall only use it for $\gamma \in [0,\half]$.
For $\gamma \in [\half,1]$ we shall have to consider another function, namely
$\varphi(\gamma,s) = \frac{\ln(1-\gamma^s)}{\ln(1-\gamma)}$,
which is, for the time being,
only defined for $\gamma \in (0,1)$ and $s > 0$.
We calculate its limit as $\gamma \to 1$ for a fixed $s > 0$ making
several uses of l'H\^opital's rule (marked with
$*$):
\begin{align*}
  \lim_{\gamma \to 1} \frac{\ln(1-\gamma^s)}{\ln(1-\gamma)}
  & =^* \lim_{\gamma\to1} \frac{-s\gamma^{s-1}(1-\gamma^s)^{-1}}{-(1-\gamma)^{-1}}
  = \lim_{\gamma\to1} \frac{s\gamma^{s-1}(1-\gamma)}{1-\gamma^s} \\
  & =^* \lim_{\gamma\to1} \frac{s(s-1)\gamma^{s-2} - s^2\gamma^{s-1}}{-s\gamma^{s-1}}
  = \frac{-s}{-s} = 1.
\end{align*}
It is therefore natural to extend $\varphi$ by $\varphi(1,s) = 1$.
This function is continuous in each variable for $s > 0$ and
$\gamma \in (0,1]$, and we wish to show that it is continuous as a
function of two variables.
In fact, all we need is to show it is continuous on
$[\half,1] \times [1,r]$.

Assume $\gamma \in (0,1)$, $s \in [1,r]$.
A straightforward verification leads to:
\begin{gather*}
  \frac{\ln(1-\gamma^r)}{\ln(1-\gamma)}
  \leq \frac{\ln(1-\gamma^s)}{\ln(1-\gamma)} \leq 1,
  \intertext{whereby:}
  \left| 1-\frac{\ln(1-\gamma^s)}{\ln(1-\gamma)} \right|
  \leq \left| 1-\frac{\ln(1-\gamma^r)}{\ln(1-\gamma)} \right|.
\end{gather*}
Thus $\lim_{\gamma \to 1} \frac{\ln(1-\gamma^s)}{\ln(1-\gamma)} = 1$ \emph{uniformly}
for $s \in [1,r]$, and $\varphi(\gamma,s)$ is indeed continuous
on $[\half,1] \times [1,r]$.

We now define for $1 \leq s \leq r$:
\begin{align*}
  A_s & = \inf \left\{
    \frac{\ln(1-\gamma^s)}{s\ln(1-\gamma)}\colon \gamma \in [\half,1)
  \right\} \leq \frac{1}{s}, \\
  B_s^- & = \sup \left\{
    \frac{1-\gamma^t}{(1-\gamma)^t}\colon \gamma \in [0,\half], t\in[1/s,s]
  \right\}, \\
  B_s^+ & = \sup \left\{
    \frac{1+\gamma^t}{(1+\gamma)^t}\colon \gamma \in [0,1], t\in[1/s,s]
  \right\}, \\
  B_s & = \max\{ B_s^-, B_s^+ \}.
\end{align*}
By continuity of $\varphi(\gamma,s)$, and since $\varphi(\gamma,1) = 1$ for all $\gamma$:
$\lim_{s\to1} A_s = \lim_{s\to1} \frac{1}{s} = 1$.
Similarly $\lim_{s\to1} B_s = 1$.

In particular we have for $s \in [1,r]$ and $\gamma \in [\half,1)$:
$A_s \leq \frac{\ln(1-\gamma^s)}{s\ln(1-\gamma)}$ whereby
$sA_s\ln(1-\gamma) \geq \ln(1-\gamma^s)$ and thus
$(1-\gamma)^{sA_s} \geq 1-\gamma^s$.

\begin{lem}
  Let $\alpha,\beta \in [-1,1]$ and $1/s \leq t \leq s$.
  Then
  \begin{gather*}
    \big| \sgn(\alpha)|\alpha|^t - \sgn(\beta)|\beta|^t \big| \leq
    \max \Big\{ |\alpha-\beta|^{A_st}, B_s|\alpha-\beta|^t \Big\}.
  \end{gather*}
\end{lem}
\begin{proof}
  We may assume that $|\alpha| \geq |\beta|$ by symmetry.
  We may further assume that $\alpha,\beta \neq 0$.
  Assume first that $\sgn(\alpha\beta) = -1$.
  Then:
  \begin{align*}
    \big| \sgn(\alpha)|\alpha|^t - \sgn(\beta)|\beta|^t \big|
    & = |\alpha|^t (1 + |\beta/\alpha|^t) \\
    & \leq |\alpha|^t B_s(1 + |\beta/\alpha|)^t \\
    & = B_s |\alpha-\beta|^t.
  \end{align*}
  A similar argument shows that when $\sgn(\alpha\beta) = 1$ and
  $|\beta/\alpha| \leq 1/2$:
  \begin{align*}
    \big| \sgn(\alpha)|\alpha|^t - \sgn(\beta)|\beta|^t \big|
    & \leq B_s |\alpha-\beta|^t.
  \end{align*}
  Finally, assume $\sgn(\alpha\beta) = 1$ and $|\beta/\alpha| \geq 1/2$.
  We use the fact that $|\alpha| \leq 1$ and $A_s \leq 1/s < 1$ imply that
  $|\alpha| \leq |\alpha|^{A_s}$:
  \begin{align*}
    \big| \sgn(\alpha)|\alpha|^t - \sgn(\beta)|\beta|^t \big|
    & = |\alpha|^t (1 - |\beta/\alpha|^t)
    \leq |\alpha|^t (1-|\beta/\alpha|^s) \\
    & \leq |\alpha|^t (1 - |\beta/\alpha|)^{A_ss} \\
    & \leq |\alpha|^{A_st} (1 - |\beta/\alpha|)^{A_st} \\
    & = |\alpha-\beta|^{A_st}.
  \end{align*}
  This completes the proof.
\end{proof}

\begin{lem}
  For all $\gamma,t \in [0,1]$ : $t(1-\gamma) + \gamma^t \leq 1$
  (where $0^0 = 1$).
\end{lem}
\begin{proof}
  This is clear for $t \in \{0,1\}$.
  So let $t \in (0,1)$, and let
  $f_t(\gamma) = t(1-\gamma) + \gamma^t$.
  Then $f_t(1) = 1$, and for $0 < \gamma < 1$ and $t-1 < 0$ we have
  $\gamma^{t-1} > 1$ whereby:
  \begin{gather*}
    \frac{d}{d\gamma}f_t = -t+t\gamma^{t-1} >-t+t = 0.
  \end{gather*}
  Thus $f_t(\gamma) \leq 1$ for all $\gamma \in [0,1]$.
\end{proof}

For $1 \leq s \leq r$ and $0 \leq x \leq 2$, define:
\begin{gather*}
  \eta_s(x) =
  \begin{cases}
    x^{A_s/s} & x \leq 1 \\
    x^s & 1 < x \leq 2
  \end{cases} \\
  \hat \eta_s(x) = 2^{1-A_s}B_s\eta_s(x)/A_s.
\end{gather*}

\begin{lem}
  \label{lem:EtaConvId}
  As $s \to 1$, the functions $\eta_s$ converge uniformly to the identity.
  As a consequence, $\hat \eta_s \to \id$ uniformly as $s \to 1$.
\end{lem}
\begin{proof}
  For $\eta_s$, one verifies uniform convergence separately for
  $x \mapsto x^{A_s/s}$ on $[0,1]$ and for $x \mapsto x^s$ on $[1,2]$.
  Uniform convergence of $\hat \eta_s$ follows.
\end{proof}

For $1 \leq s \leq r$, define:
\begin{align*}
  C_s^1 & = \sup \left\{
    \left|
      \sgn\left( \half[\alpha+\beta] \right)
      \left| \half[\alpha+\beta] \right|^t
      -
      \half[\sgn(\alpha)|\alpha|^t + \sgn(\beta)|\beta|^t]
    \right|
    \colon
    \alpha,\beta \in [-1,1], t \in [1/s,s]
  \right\}, \\
  C_s^2 & = \sup \bigl\{ |x - \hat \eta_s(x)|\colon x \in [0,2] \bigr\}, \\
  C_s & = \max \{C_s^1, C_s^2\}.
\end{align*}
Then $\lim_{s\to1} C_s = 0$.

\subsection{Perturbing the exponent}

\begin{dfn}
  \label{dfn:RhoFunc}
  Let $(X,\fB,\mu)$ be a measure space and
  $p,q\colon X \to [1,r]$ measurable.
  We define $\sE_{p,q}\colon L_0(X,\fB,\mu) \to L_0(X,\fB,\mu)$ by:
  \begin{gather*}
    (\sE_{p,q} f)(x) = \sgn(f(x)) |f(x)|^{p(x)/q(x)}.
  \end{gather*}
\end{dfn}

\begin{lem}
  \label{lem:ExpPertProps}
  We continue with the assumptions of \fref{dfn:RhoFunc}.
  Let $(N,\Theta) = (L_{p(\cdot)}(X,\fB,\mu),\Theta_{p(\cdot)})$ and
  $(N',\Theta') = (L_{q(\cdot)}(X,\fB,\mu),\Theta_{q(\cdot)})$.
  \begin{enumerate}
  \item For each $f \in L_0(X,\fB,\mu)$ we have
    $\Theta(f) = \Theta'(\sE_{p,q} f)$.
    Thus in particular $\sE_{p,q}$ sends $N$ into $N'$ and the unit
    ball of $N$ into the unit ball of $N'$.
  \item The mapping $\sE_{p,q}$ is bijective, its inverse being
    $\sE_{q,p}$.
    It restricts to a bijection between $N$ and $N'$, as well as to a
    bijection between their respective unit balls.
  \item
    \label{item:ExpPertDenseChangeCommute}
    The mapping $\sE_{p,q}$ commutes with measure density change.
    More precisely, assume $\nu$ is another measure on $(X,\fB)$,
    equivalent to $\mu$, say $d\nu(x) = \zeta(x)d\mu(x)$.
    Let $M = L_{p(\cdot)}(X,\fB,\nu)$, $M' = L_{q(\cdot)}(X,\fB,\nu)$.
    Let $D^p_{\mu,\nu}\colon N\to M$ and $D^{q}_{\mu,\nu}\colon N' \to M'$ be the
    respective density change mappings.
    Then $D^{q}_{\mu,\nu} \circ \sE_{p,q}
    = \sE_{p,q} \circ D^p_{\mu,\nu}\colon N \to M'$.
  \end{enumerate}
\end{lem}
\begin{proof}
  For the first item we calculate that:
  \begin{gather*}
    \Theta'(\sE_{p,q} f) = \int |f(x)|^{p(x)} d\mu = \Theta(f).
  \end{gather*}
  The second item follows.
  Finally, we calculate:
  \begin{align*}
    (D^{q}_{\mu,\nu} \sE_{p,q} f)(x)
    & = \zeta(x)^{-1/q(x)}(\sE_{p,q}f)(x) \\
    & = \zeta(x)^{-1/q(x)}\sgn(f(x))|f(x)|^{p(x)/q(x)} \\
    & = \sgn(\zeta(x)^{-1/p(x)}f(x))|\zeta(x)^{-1/p(x)}f(x)|^{p(x)/q(x)} \\
    & = \sgn((D^p_{\mu,\nu} f)(x))|(D^p_{\mu,\nu} f)(x)|^{p(x)/q(x)} \\
    & = (\sE_{p,q} D^p_{\mu,\nu} f)(x),
  \end{align*}
  proving the third item.
\end{proof}

\begin{prp}
  \label{prp:RhoDistChange}
  We continue with the notation and assumptions of
  \fref{lem:ExpPertProps}.
  Assume that $s$ is such that $1/s \leq q(x)/p(x)\leq s$
  (for example, we can take $s = r$).
  Then for every $f,g \in N_1$ (the unit ball of $N$):
  $\|\sE_{p,q} f-\sE_{p,q} g\| \leq \hat \eta_s(\|f-g\|)$
  and $\|f-g\| \leq \hat \eta_s(\|\sE_{p,q} f - \sE_{p,q} g\|)$.
\end{prp}
\begin{proof}
  Let $f,g \in N_1$.
  By \fref{lem:ExpPertProps}\ref{item:ExpPertDenseChangeCommute} we may
  assume that $|f|\vee|g| = \chi_S$ for some set $S \in \fB$, so
  $f(x),g(x) \in [-1,1]$.
  Let
  \begin{align*}
    h(x) & = \big| \sgn(f(x))|f(x)|^{p(x)/q(x)} -
    \sgn(g(x))|g(x)|^{p(x)/q(x)} \big| \\
    S_1 & = \left\{
      x \in S\colon h(x) \leq |f(x)-g(x)|^{A_sp(x)/q(x)}
    \right\} \\
    S_2 & = S \setminus S_1 \subseteq \left\{
      x \in S\colon h(x) \leq B_s|f(x)-g(x)|^{p(x)/q(x)}
    \right\}.
  \end{align*}
  We observe that as $\|f\|,\|g\| \leq 1$ we have $\mu(S) \leq 2$.
  Observe also that
  $A_sq(x)/s \leq q(x)/s \leq p(x)$
  and that
  $sq(x)/A_s \geq sq(x) \geq p(x)$.
  It follows that if $\|f-g\| \leq 1$ then:
  \begin{align*}
    \|f-g\|^{p(x)}
    & \leq \eta_s( \|f-g\| )^{q(x)/A_s} = \|f-g\|^{q(x)/s} \\
    & \leq \eta_s( \|f-g\| )^{q(x)} = \|f-g\|^{A_sq(x)/s}.
    \intertext{Otherwise $1 < \|f-g\| \leq 2$, and:}
    \|f-g\|^{p(x)}
    & \leq \eta_s( \|f-g\| )^{q(x)} = \|f-g\|^{sq(x)} \\
    & \leq \eta_s( \|f-g\| )^{q(x)/A_s} = \|f-g\|^{sq(x)/A_s}.
  \end{align*}

  Let $\gamma = \int_{S_1} \frac{ |f(x) - g(x)|^{p(x)} }{\|f-g\|^{p(x)}} d\mu(x)$
  and $a = \hat \eta_s( \|f-g\| ) = 2^{1-A_s}B_s\eta_s(\|f-g\|)/A_s$.
  Then:
  \begin{align*}
    \Theta'\left( \frac{\sE_{p,q} f-\sE_{p,q} g}{a} \right)
    & = \int_S \frac{h(x)^{q(x)}}{a^{q(x)}} d\mu(x) \\
    & \leq \int_{S_1} \frac{|f(x)-g(x)|^{A_sp(x)}}{a^{q(x)}} d\mu(x)
    + \int_{S_2} \frac{B_s^{q(x)}|f(x)-g(x)|^{p(x)}}{a^{q(x)}} d\mu(x)
  \end{align*}
  We work on each integral separately.
  \begin{align*}
    \int_{S_1} \frac{|f(x)-g(x)|^{A_sp(x)}}{a^{q(x)}} d\mu(x)
    & =
    \int_{S_1} \frac{\mu(S_1)A_s^{q(x)}}{(2^{1-A_s}B_s)^{q(x)}} \left(
      \frac{|f(x)-g(x)|^{p(x)}}{\eta_s( \|f-g\| )^{q(x)/A_s}}
    \right)^{A_s} \frac{d\mu(x)}{\mu(S_1)} \\
    & \leq \frac{\mu(S_1)}{2^{1-A_s}}
    \int_{S_1} \left(
      \frac{|f(x)-g(x)|^{p(x)}}{\|f-g\|^{p(x)}}
    \right)^{A_s} \frac{d\mu(x)}{\mu(S_1)} \\
    & \leq \frac{\mu(S_1)}{2^{1-A_s}} \left(
      \int_{S_1} \frac{|f(x)-g(x)|^{p(x)}}{\|f-g\|^{p(x)}} \frac{d\mu(x)}{\mu(S_1)}
    \right)^{A_s} \\
    & = \frac{\mu(S)^{1-A_s}}{2^{1-A_s}} \gamma^{A_s} \leq \gamma^{A_s}.
  \end{align*}
  And:
  \begin{align*}
    \int_{S_2} \frac{B_s^{q(x)}|f(x)-g(x)|^{p(x)}}{a^{q(x)}} d\mu(x)
    & =
    \int_{S_2} \frac{(A_sB_s)^{q(x)}}{(2^{1-A_s}B_s)^{q(x)}}
    \frac{|f(x)-g(x)|^{p(x)}}{\eta_s( \|f-g\| )^{q(x)}} d\mu(x) \\
    & \leq A_s \int_{S_2} \frac{|f(x)-g(x)|^{p(x)}}{\|f-g\|^{p(x)}} d\mu(x) = A_s(1-\gamma).
  \end{align*}
  Thus:
  \begin{gather*}
    \Theta'\left( \frac{\sE_{p,q} f-\sE_{p,q} g}{a} \right)
    \leq \gamma^{A_s} + A_s(1-\gamma) \leq 1
  \end{gather*}
  We conclude that
  $\|\sE_{p,q} f - \sE_{p,q} g\| \leq a = \hat \eta_s( \|f-g\| )$.
  Since $1/s \leq p(x)/q(x)\leq s$ as well we have
  $\|f-g\| = \|\sE_{q,p}\sE_{p,q} f - \sE_{q,p}\sE_{p,q} g\|
  \leq \hat \eta_s(\|\sE_{p,q} f - \sE_{p,q} g\|).$
\end{proof}

\begin{cor}
  The mapping $\sE_{p,q}\colon N_1 \to N'_1$ is uniformly continuous, the modulus of
  uniform continuity depending solely on $r$.
\end{cor}
\begin{proof}
  Define $\Delta_r(\varepsilon) = \min \left\{
    \left( 2^{A_r-1}A_r\varepsilon/B_r \right)^{r/A_r}, 1
  \right\}$.
  Then for all $\varepsilon > 0$ we have $\Delta_r(\varepsilon) > 0$ and
  $\|f-g\| < \Delta_r(\varepsilon) \Longrightarrow \|\sE_{p,q} f - \sE_{p,q} g\| \leq \varepsilon$.
\end{proof}

\begin{prp}
  \label{prp:RhoChanges}
  Let $\sE_{p,q}\colon N \to N'$ be as in \fref{dfn:RhoFunc}, and let
  $f,g \in N_1$.
  Then:
  \begin{enumerate}
  \item $\sE_{p,q}0 = 0$; $-\sE_{p,q} f = \sE_{p,q}(-f)$;
    $\sE_{p,q}(|f|) = \bigl|\sE_{p,q} f\bigr|$.
  \item $\bigl| \|f-g\| - \|\sE_{p,q} f - \sE_{p,q} g\| \bigr| \leq C_s$.
  \item $\|\sE_{p,q}\half[f+g] - \half[\sE_{p,q} f+\sE_{p,q} g]\| \leq 2C_s$.
  \end{enumerate}
\end{prp}
\begin{proof}
  The first item is clear.
  For the second we use \fref{prp:RhoDistChange}:
  \begin{align*}
    \|\sE_{p,q} f - \sE_{p,q} g\| - \|f-g\| & \leq \hat \eta_s( \|f-g\| ) - \|f-g\| \leq C_s, \\
    \|f-g\| - \|\sE_{p,q} f - \sE_{p,q} g\| & \leq \hat \eta_s( \|\sE_{p,q} f - \sE_{p,q} g\| ) - \|\sE_{p,q} f - \sE_{p,q} g\| \leq C_s.
  \end{align*}
  We may assume that $|f|\vee|g| = \chi_S$ for some measurable
  set $S$, so $\mu(S) \leq 2$.
  By definition of $C_s$ we have
  $
  \left|
    \sE_{p,q}\half[f+g](x) - \half[\sE_{p,q} f+\sE_{p,q} g](x)
  \right|
  \leq C_s
  $
  for $x \in S$, and we get:
  \begin{align*}
    \Theta\left( \frac{\sE_{p,q}\half[f+g] - \half[\sE_{p,q} f+\sE_{p,q} g]}{2C_s} \right)
    & \leq \int_S 2^{-q(x)} \, d\mu(x) \leq \half[\mu(S)] \leq 1.
  \end{align*}
  The third item follows.
\end{proof}

We now wish to define a perturbation system $\fp$ for
$\cL_{Bl}$-structures.
We do this by defining a $\fp(\varepsilon)$-perturbation of structures $N$
and $N'$ directly as a bijection $\theta\colon N \to N'$ such that for all
$f,g,h \in N$:
\begin{gather*}
  \theta0 = 0, \\
  \theta(-f) = -\theta f, \\
  \theta(|f|) = |\theta f|, \\
  \left| d\left( \half[f+g],h \right)
    - d\left( \half[\theta f+\theta g],\theta h \right) \right|
  \leq \varepsilon, \\
  \big| \|f\| - \|\theta f\| \big| \leq \varepsilon,
  \intertext{and:}
  e^{-\varepsilon e^\varepsilon} d(f,g)^{e^\varepsilon} \leq d(\theta f,\theta g)
  \leq e^\varepsilon d(f,g)^{e^{-\varepsilon}}.
\end{gather*}
(While for most symbols we can just allow to ``change by $\varepsilon$'', we need
to take special care with the distance symbol.)
This indeed defines a perturbation system, as it clearly verifies the
following characterisation:
\begin{fct}
  \label{fct:PertChar}
  Let $T$ be a theory, and assume that
  for each $r \in \bR^+$ and $M,N \in \Mod(T)$,
  $\Pert'_r(M,N)$ is a set of bijections of $M$ with $N$
  satisfying the following properties:
  \begin{enumerate}
  \item Monotonicity:
    $\Pert'_r(M,N) = \bigcap_{s>r} \Pert'_s(M,N)$.
  \item Non-degenerate reflexivity:
    $\Pert'_0(M,N)$ is the set of isomorphisms
    of $M$ with $N$.
  \item Symmetry:
    $f \in \Pert'_r(M,N)$ if and only $f^{-1} \in \Pert'_r(N,M)$.
  \item Transitivity:
    if
    $f \in \Pert'_r(M,N)$ and
    $g \in \Pert'_s(N,L)$ then
    $g \circ f \in \Pert'_{r+s}(M,L)$.
  \item Uniform continuity:
    for each $r \in \bR^+$, all members of $\Pert'_r(M,N)$, where
    $M,N$ vary over all models of $T$, satisfy a common modulus of
    uniform continuity.
  \item Ultraproducts:
    If $f_i \in \Pert'_r(M_i,N_i)$ for $i \in I$, and $\sU$ is
    an ultrafilter on $I$ then
    $\prod_\sU f_i \in \Pert'_r\bigl( \prod_\sU M_i, \prod_\sU N_i \bigr)$.
    (Note that $\prod_\sU f_i$ exists by the uniform continuity
    assumption).
  \item Elementary substructures:
    If $f \in \Pert'_r(M,N)$, $M_0 \preceq M$, and
    $N_0 = f(M_0) \preceq N$ then
    $f\rest_{M_0} \in \Pert'_r(M_0,N_0)$.
  \end{enumerate}
  Then there exists a unique perturbation system $\fp$ for $T$ such
  that $\Pert'_r(M,N) = \Pert_{\fp(r)}(M,N)$ for all $r$, $M$ and $N$.
\end{fct}
\begin{proof}
  \cite[Theorem~4.4]{BenYaacov:TopometricSpacesAndPerturbations}.
\end{proof}

Recall that given two $n$-types $p,q$ we say that
$d_\fp(p,q) \leq \varepsilon$ if there are
$\cL_{Bl}$-structures $N,N'$ and an $\varepsilon$-perturbation $\theta\colon N \to N'$ sending
a realisation of $p$ to one of $q$.

\begin{lem}
  \label{lem:PertEpsilonS}
  For every $\varepsilon > 0$ there exists $s > 1$ such that if
  $N = L_{p(\cdot)}(X,\fB,\mu)$, $N' = L_{q(\cdot)}(X,\fB,\mu)$ and
  $\sE_{p,q}\colon N \to N'$ is as in \fref{dfn:RhoFunc}
  (so in particular $1/s \leq p(x)/q(x) \leq s$ for almost all $x \in X$),
  then $\sE_{p,q}$ is a $\fp(\varepsilon)$-perturbation.
\end{lem}
\begin{proof}
  By \fref{prp:RhoDistChange},
  \fref{prp:RhoChanges} and the fact that $\lim_{s \to 1} C_s = 0$.
\end{proof}

\begin{lem}
  \label{lem:FiniteApprox}
  Fix a compact $K \subseteq [1,r]$ and $s > 1$.
  Then there is a finite set $K_s \subseteq [0,1]$ such that for every
  atomless measure space $(X,\fB,\mu)$ and
  $p\colon X \to [1,r]$ with $\essrng(p) = K$ there exists
  $q\colon X \to [1,r]$ such that $\essrng(q) = K_s$ and
  for almost all $x \in X$: $1 \leq q(x)/p(x) \leq s$.
\end{lem}
\begin{proof}
  By compactness we can cover $K$ with
  finitely many open intervals
  $[1,r] \subseteq \bigcup \{(a_i,b_i)\colon i < n\}$, with $1 < b_i/a_i \leq s$.
  We may assume that $K\cap(a_i,b_i) \neq \emptyset$ for all $i < n$.
  We then define $K_s = \{b_i\colon i < n\}$.

  Assume now that $(X,\fB,\mu)$ is atomless and $p\colon X \to [1,r]$ satisfies
  $\essrng(p) = K$.
  We can then split $X$ into a finite disjoint union of positive
  measure sets $X = \bigcup_{i<n} X_i$ such that the essential range of
  $p_i = p\rest_{X_i}$ is contained in $(a_i,b_i)$.
  Define $q(x) = b_i$ when $x \in X_i$.
  Then $q$ is as required.
\end{proof}

\begin{fct}
  \label{fct:LpCatStab}
  For $K$ consisting of a single point, the theory
  $\Th(\cA\cN_K)$ is $\aleph_0$-categorical and $\aleph_0$-stable.
\end{fct}
\begin{proof}
  \cite{BenYaacov-Berenstein-Henson:LpBanachLattices}.
\end{proof}

\begin{lem}
  \label{lem:FiniteEssRngCatStab}
  Let $K \subseteq [1,r]$ be finite.
  Then $\Th(\cA\cN_K)$ is $\aleph_0$-categorical and $\aleph_0$-stable.
\end{lem}
\begin{proof}
  Let $K = \{p_i\colon i < n\}$, $p_0 < \ldots < p_{n-1}$.
  If $N = L_{p(\cdot)}(X,\fB,\mu) \in \cA\cN_K$
  then $X$ can be written as a disjoint union
  $X = \bigcup_{i<n}X_i$ where $X_i \in \fB$, $\mu(X_i) > 0$
  and $p\rest_{X_i} \equiv p_i$ a.e.
  For $i < n$ let $N_i$ be the Banach lattice $\chi_{X_i}N$.
  Thus the $N_i$ are orthogonal bands in $N$ and
  $N = \bigoplus_{i<n} N_i$.
  Since we can recover $\Theta$ from the norm on each $N_i$
  we can recover $\Theta$ on $N$,
  and thus we can recover the norm on $N$.
  Similarly, as the $N_i$ are orthogonal bands we can recover the
  lattice structure on $N$ from that of $N_i$.

  Now, if $N$ is separable (and atomless), each $N_i$ is separable
  and atomless, and thus uniquely determined by $p_i$ up to
  isomorphism, whereby $N$ is uniquely
  determined by $K$.
  This proves $\aleph_0$-categoricity.

  Similarly, let $N' \preceq N$ be a separable elementary sub-model and
  let $N'_i = N' \cap N_i$.
  By $\aleph_0$-stability of $\Th(N_i)$, $\tS^{N_i}_\ell(N'_i)$ is metrically
  separable for each $i$.
  Now let $\bar f = f^0,\ldots,f^{\ell-1} \in N$, and let
  $f^j = \sum_{i<n} f^j_k$ where $f^j_i \in N_i$.
  Naming $\Theta$ and using quantifier elimination we see that
  $\tp^N(\bar f/N')$ is uniquely determined by
  $(\tp^{N_i}(\bar f_i/N'_i)\colon i < n)$, and we might as well write
  $\tp^N(\bar f/N') = \sum_{i<n} \tp^{N_i}(\bar f_i/N'_i)$.
  If $q = \sum_{i<n} q_i$ and $q' = \sum_{i<n} q'_i$ are two such
  decompositions then we have
  $d(q,q') \leq \sum_{i<n} d(q_i,q'_i)$.
  Thus $\tS^N_\ell(N')$ is metrically separable.
\end{proof}

We can now conclude:
\begin{thm}
  \label{thm:PertCatStab}
  The theory $\Th(\cA\cN_{\subseteq[1,r]})$ is $\fp$-$\aleph_0$-stable, and every
  completion thereof (which is of the form $\Th(\cA\cN_K)$)
  is $\fp$-$\aleph_0$-categorical.
\end{thm}
\begin{proof}
  Combining \fref{lem:PertEpsilonS} and \fref{lem:FiniteApprox} we see
  that for every $\varepsilon > 0$ there is a finite set $K' \subseteq [1,r]$ such that
  every separable $N,N' \in \cA\cN_K$ admit $\fp(\varepsilon/2)$-perturbations with
  separable $\tilde N,\tilde N \in \cA\cN_{K'}$, respectively.
  But $\tilde N \cong \tilde N'$ by \fref{lem:FiniteEssRngCatStab}, so $N$
  and $N'$ admit a $\fp(\varepsilon)$-perturbation.

  Similarly for $\fp$-$\aleph_0$-stability.
\end{proof}

\begin{cor}
  The theory $\Th(\cA\cN_{\subseteq[1,r]})$ is stable.
\end{cor}
\begin{proof}
  By
  \cite[Proposition~4.11]{BenYaacov:TopometricSpacesAndPerturbations}
  $\lambda$-$\fp$-stability
  implies stability.
  (See \cite[Section~4.3]{BenYaacov:TopometricSpacesAndPerturbations}
  for more properties and
  characterisations of $\aleph_0$-stability up to perturbation.)
\end{proof}

\begin{rmk*}
  It is in fact also true that the theory $\Th(\cN_{\{p_0\}})$ (i.e.,
  constant $p$, but possibly with atoms) is $\aleph_0$-stable,
  although this fact is not proved anywhere in the literature
  at the time of writing.
  By the same reasoning, the theory $\Th(\cN_{\subseteq[1,r]})$ is
  $\fp$-$\aleph_0$-stable and in particular stable.
\end{rmk*}

\appendix
\section{Some basic continuous model theory}
\label{apx:ContModTh}

\subsection{Definability and monotonicity}

\begin{thm}[Beth's definability theorem for continuous logic]
  \label{thm:Beth}
  Let $\cL_0 \subseteq \cL$ be continuous signatures with the same
  sorts (i.e., $\cL$ does not add new sorts on top of those existing
  in $\cL_0$) and $T$ an $\cL$-theory
  such that every $\cL_0$-structure $M_0$ admits at most a single
  expansion to an $\cL$-structure $M$ which is a model of $T$.
  Then every symbol in $\cL$ admits an explicit $\cL_0$-definition in
  $T$.
  That is to say that for every predicate symbol $P(\bar x) \in \cL$
  is equal in all models of $T$ to some $\cL_0$-definable predicate
  $\varphi_P(\bar x)$,
  and for every function symbol $f(\bar x) \in \cL$
  the predicate $d(f(\bar x),y)$
  is equal in all models of $T$ to some $\cL_0$-definable predicate
  $\varphi_f(\bar x)$.
\end{thm}
\begin{proof}
  For convenience we shall assume that the language is single sorted,
  but the same proof holds for a many sorted language.

  Let $P \in \cL$ be an $n$-ary function symbol, and consider the
  mapping $\theta_n\colon \tS_n(T) \to \tS_n(\cL_0)$,
  the latter being the space
  of all complete $n$-types in the language $\cL_0$.
  It is known that $\theta_n$ is continuous, and we claim it is injective.

  Indeed, let $p,p' \in \tS_n(T)$ be such that
  $\theta_n(p) = \theta_n(p') = q$.
  Let $M \models p(\bar a)$ and $M' \models p'(\bar a')$, so
  Then $\tp^{\cL_0}(\bar a) = \tp^{\cL_0}(\bar a') = q$.
  \begin{clm}
    There exists an elementary extension $M \preceq M_1$ and an
    $\cL_0$-elementary embedding $M' \hookrightarrow M_1$
    sending $\bar a'$ to
    $\bar a$.
  \end{clm}
  \begin{clmprf}
    We need to verify that
    $\Th_{\cL(M)}(M) \cup \Th_{\cL_0}(M') \cup \{\bar a = \bar a'\}$ is
    consistent.
    But the assumptions on the types tell us precisely that
    $\Th_{\cL_0}(M') \cup \{\bar a = \bar a'\}$ is approximately finitely
    satisfiable in $(M,\bar a)$.
  \end{clmprf}

  we shall identify $M'$ as a set with its image in $M_1$, and in
  particular assume that $\bar a = \bar a'$.

  \begin{clm}
    Let $N$ and $N'$ be two $\cL$-structures,
    and assume that $N \preceq_{\cL_0} N'$ (but needn't even be an
    $\cL$-substructure).
    Then there exists $N'' \succeq N$ such that $N' \preceq_{\cL_0} N''$.
  \end{clm}
  \begin{clmprf}
    The assumption $N \preceq_{\cL_0} N'$ implies that $\Th_{\cL_0(N')}(N')$
    is approximately finitely satisfiable in $N$,
    so $\Th_{\cL(N)}(N) \cup \Th_{\cL_0(N')}(N')$ is consistent.
  \end{clmprf}

  Using the claim we can extend the pair $M' = M'_0 \preceq_{\cL_0} M_1$ to
  a chain of $\cL$-structures
  We now construct a sequence of structures
  $M'_0
  \preceq_{\cL_0} M_1
  \preceq_{\cL_0} M'_1
  \preceq_{\cL_0} M_2
  \preceq_{\cL_0} M'_2 \ldots$
  such that $M_i \preceq M_{i+1}$ and $M'_i \preceq M'_{i+1}$.

  Let $M_\omega = \bigcup M_i$, $M'_\omega = \bigcup M'_i$.
  Then both $M_\omega$ and $M'_\omega$ are models of $T$ and have the same
  $\cL_0$-reduct, and are therefore the same.
  It follows that
  $p = \tp^{M_\omega}(\bar a) = \tp^{M'_\omega}(\bar a) = p'$.

  Once we have established that $\theta_n$ is an injective continuous
  mapping between compact Hausdorff spaces it is necessarily an
  embedding (i.e., a homeomorphism with its image).
  We may identify the predicate $P$ with a continuous function
  $P\colon \tS_n(T) \to [0,1]$.
  By Tietze's extension theorem there exists a continuous function
  $\varphi_P \colon \tS_n(\cL_0) \to [0,1]$ such that
  $P = \varphi_P \circ \theta_n$.
  Then $\varphi_P$ is the required $\cL_0$-definable predicate.

  If $f$ is a function symbol, apply the preceding argument to
  $d(f(\bar x),y)$.
\end{proof}

\begin{dfn}
  \label{dfn:IncDecFormula}
  Let $T$ be a theory, $\varphi(\bar x)$ a definable predicate.
  We say that $\varphi$ is \emph{increasing} (\emph{decreasing}) in $T$
  if whenever
  $M \subseteq N$ are both models of $T$ and $\bar a \in M$ we have
  $\varphi(\bar a)^M \leq \varphi(\bar a)^N$
  ($\varphi(\bar a)^M \geq \varphi(\bar a)^N$).
  We say that $\varphi$ is \emph{constant} in $T$ if it is both increasing
  and decreasing in $T$.
\end{dfn}

\begin{dfn}
  \label{dfn:SupInfFormula}
  A \emph{$\sup$-formula} is a formula of the form
  $\sup_{\bar y} \varphi(\bar x,\bar y)$ where $\varphi$ is
  quantifier-free.

  A \emph{$\sup$-definable predicate} is a definable predicate
  which can be written syntactically as
  $\flim \varphi_n(\bar x)$ where each $\varphi_n$ is a
  $\sup$-formula.
  (See \cite[Definition~3.6]{BenYaacov-Usvyatsov:CFO} and subsequent
  discussion for the definition and properties of the forced limit
  operation $\flim$.)
  Notice that every such predicate is equal to a uniform limit of
  $\sup$-formulae.

  We make the analogous definitions for $\inf$.
\end{dfn}

\begin{thm}
  \label{thm:MonForm}
  Let $T$ be a theory, $\varphi(\bar x)$ a definable predicate.
  Then $\varphi$ is increasing (decreasing) in $T$ if and only if
  $\varphi$ is equivalent modulo $T$ to a $\sup$-definable
  ($\inf$-definable) predicate.
\end{thm}
\begin{proof}
  Clearly it suffices to prove the case of increasing definable
  predicates.
  Right to left being immediate, we prove left to right.

  Assume therefore that $\varphi(\bar x)$ is increasing in $T$.
  Let $\Psi$ be the collection of all $\sup$-formulae
  $\psi(\bar x) = \sup_{\bar y} \tilde \psi(\bar x,\bar y)$ such that
  $T \vdash \psi(\bar x) \leq \varphi(\bar x)$.
  Notice that the latter means that
  $T \vdash \tilde \psi(\bar x,\bar y) \leq \varphi(\bar x)$.
  If for every $n < \omega$ there is $\psi_n \in \Psi$ such that
  $T \vdash \varphi(\bar x) \dotminus 2^{-n} \leq \psi(\bar x)$ then
  $\varphi = \flim \psi_n$ and we are done.
  In order to conclude we shall assume the converse and obtain a
  contradiction.

  We assume then that there is $n < \omega$ such that
  $T\cup \{\varphi(\bar x) \dotminus \psi(\bar x) \geq 2^{-n}\}$
  is consistent for all
  $\psi \in \Psi$.
  As $\Psi$ is closed under $\vee$ and
  $\varphi \dotminus (\psi\vee\psi') \geq 2^{-n}
  \Longrightarrow \varphi \dotminus \psi \geq 2^{-n}$,
  the set
  $\Sigma = T\cup \{\varphi\dotminus \psi\geq2^{-n}\}_{\psi \in \Psi}$
  is consistent.
  Let $(M,\bar a)$ be a model for it, and let
  $r = \varphi(\bar a)^M$.

  Let
  $\Sigma'
  = T \cup \Diag_a(M)
  \cup \{\varphi(\bar a) \leq r-2^{-n}\}$.
  Here $\Diag_a(M)$ denotes the atomic diagram of $M$, namely the
  family of all conditions of the form
  $\chi(\bar a) = \chi(\bar a)^M$ where
  $\chi(\bar x)$ is an atomic formula and $\bar a \in M$,
  so a model of $\Diag_a(M)$ is a structure in which $M$ is embedded.
  If $\Sigma'$ were consistent we would get a contradiction to
  $\varphi$ being increasing, so $\Sigma'$ is contradictory.
  By compactness there exists a quantifier-free formula
  $\chi(\bar x,\bar y)$ and $\bar b \in M$ such that
  $\chi(\bar a,\bar b)^M = 0$ and
  $T \cup \{\chi(\bar x,\bar y) = 0\}
  \cup \{\varphi(\bar x) \leq r-2^{-n}\}$
  is contradictory.
  It follows there is some $m$ such that
  $T \cup \{\chi(\bar x,\bar y) \leq 2^{-m}\}
  \cup \{\varphi(\bar x) \leq r-2^{-n}\}$
  is contradictory.
  Let $r' \in (r - 2^{-n},r)$ be a dyadic number, and
  let $\tilde \psi = r' \dotminus 2^m\chi$.
  Then $\tilde \psi$ is a quantifier-free formula, and
  we claim that $T \vdash \tilde \psi(\bar x,\bar y) \leq \varphi(\bar x)$.
  indeed, for any model $N \models T$ and any $\bar c,\bar d \in N$:
  \begin{align*}
    \varphi(\bar c)^N \geq r'
    & \quad \Longrightarrow \quad
    \varphi(\bar c)^N \geq r' \geq \tilde \psi(\bar c,\bar d)^N \\
    \varphi(\bar c)^N \leq r'
    & \quad \Longrightarrow \quad
    \chi(\bar c,\bar d)^N \geq 2^{-m}
    \quad \Longrightarrow \quad
    \varphi(\bar c)^N \geq 0 = \tilde \psi(\bar c,\bar d)^N.
  \end{align*}
  Thus $\psi(\bar x) = \sup_{\bar y} \tilde \psi(\bar x,\bar y) \in \Psi$,
  whereby $\varphi(\bar a)^M \dotminus \psi(\bar a)^M \geq 2^{-n}$.
  But $\chi(\bar a,\bar b)^M = 0$, so
  $\psi(\bar a) \geq r'$ whereby $\varphi(\bar a)^M \geq r'+2^{-n} > r$,
  a contradiction.
  This concludes the proof.
\end{proof}

\begin{cor}
  \label{cor:ModComp}
  A continuous theory $T$ is model complete if and only if every
  formula (definable predicate)
  is equivalent modulo $T$ to an $\inf$-definable predicate.
\end{cor}
\begin{proof}
  Left to right is by \fref{thm:MonForm}.
  For right to left, every formula $\varphi$ is decreasing in $T$, and
  considering $\neg\varphi$ every formula is increasing as well, and therefore
  constant in $T$, which means precisely that $T$ is model complete.
\end{proof}

\subsection{Interpretations}
\label{apx:Interpretations}

We turn to treat the issue of passage from one language to another in
a structure, which has arisen several times in this paper.
We start with a somewhat watered down notion of a structure being
interpretable in another.

\begin{dfn}[Interpretation schemes]
  Let $\cL_0$ and $\cL_1$ be two single sorted signatures.
  A (restricted) \emph{interpretation scheme}
  $\Phi\colon \cL_0 \to \cL_1$
  consists of a mapping assigning to every atomic $\cL_1$-formula
  $\varphi(\bar x)$ an $\cL_0$-definable predicate
  $\varphi^\Phi(\bar x)$.

  Let $M$ be an $\cL_0$-structure.
  We define $\Phi(M)$ to be any $\cL_1$-structure, should one exist,
  equipped with a mapping $\iota\colon M \to \Phi(M)$
  with a dense image, such that for every
  every atomic $\cL_1$-formula $\varphi(\bar x)$:
  \begin{gather}
    \label{eq:IntepretedFormula}
    \varphi(\iota \bar a)^{\Phi(M)} = \varphi^\Phi(\bar a)^M
    \qquad
    \text{for all } \bar a \in M.
  \end{gather}
  It is not difficult to check that the pair
  $\bigl(\Phi(M),\iota\bigr)$, if it exists,
  is unique up to a unique isomorphism, justifying the notation.
  By a convenient abuse of notation we shall omit $\iota$ altogether,
  identifying $\bar a \in M$ with $\iota \bar a \in \Phi(M)$.

  We define $\cK^\Phi$ to be the class of $\cL_0$-structures $M$ for
  which $\Phi(M)$ exists.
  More generally, if $\cK$ is a class of $\cL_1$-structures, we define
  $\Phi^{-1}(\cK) = \{M \in \cK^\Phi\colon \Phi(M)\in\cK\}$.

  By induction on the structure of $\cL_1$-formulae
  on extends the mapping $\varphi \mapsto \varphi^\Phi$
  from atomic formulae to arbitrary ones.
  If $\varphi$ is an $\cL_1$-definable predicate it can always be
  written as $\flim \varphi_n$ where $\varphi_n$ are formulae,
  and we may then define
  $\varphi^\Phi = \flim (\varphi_n)^\Phi$.
  It is straightforward to check that if
  $M \in \cK^\Phi$ then \fref{eq:IntepretedFormula} holds
  for every formula or definable predicate $\varphi$.
\end{dfn}

We qualified this notion of interpretation as ``restricted'', since
it uses the entire home sort of the interpreting structure, whereas
the tradition notion of interpretation in classical logic allows the
interpretation to take place on an arbitrary definable set.
We could extend the definition by letting the domain of the mapping
$\iota$, rather than be all of $M$, be some definable subset
$X \subseteq M^n$, where
$d(\bar x,X)$ is given uniformly by a definable predicate
$\chi^\Phi(\bar x)$ which is also prescribed by $\Phi$.
Everything we prove here regarding interpretations goes through with
this more general definition.
In particular, the class of structures in which $\chi^\Phi$ defines
the distance to a set (the zero set of $\chi^\Phi$)
is elementary.
For details on definable sets in continuous logic and their
properties we refer the reader to
\cite[Section~1]{BenYaacov:DefinabilityOfGroups}.

\begin{lem}
  Let $\Phi\colon \cL_0 \to \cL_1$ be an interpretation scheme.
  Then the class $\cK^\Phi$ is elementary and we may write
  $T^\Phi = \Th(\cK^\Phi)$.
  More generally,
  if $\cK = \Mod(T)$ is a an elementary class of $\cL_1$-structures
  then $\Phi^{-1}(\cK)$ is elementary as well,
  and we may write $\Phi^{-1}(T) = \Th\bigl( \Phi^{-1}(\cK) \bigr)$.
\end{lem}
\begin{proof}
  In the case where $\cL_1$ is purely relational,
  $T^\Phi$ merely consists of axioms expressing
  that the predicate symbols respect the uniform continuity
  moduli prescribed by $\cL_1$.
  In case there are also function symbols we need more axioms
  (all free variables are quantified universally):
  \begin{itemize}
  \item 
    Axioms expressing that $d(f(\bar x),y)^\Phi$
    respects the uniform continuity
    moduli of $f$ in the $\bar x$ and is $1$-Lipschitz in $y$.
  \item
    The axioms 
    $d(y,z)^\Phi \leq
    d(f(\bar x),y)^\Phi + d(f(\bar x),z)^\Phi$
    and
    $\inf_y\, d(f(\bar x),y)^\Phi = 0$.
    Notice that if
    $d(f(\bar x),y_n)^\Phi \to 0$ as $n \to \infty$ then
    $\{y_n\}$ is a Cauchy sequence and therefore admits
    a limit.
    Thus for all $\bar x$ there exists a unique $y$
    such that
    $d(f(\bar x),y)^\Phi = 0$,
    and for all other $z$:
    $d(f(\bar x),z)^\Phi = d(y,z)^\Phi$.
    We may then legitimately write
    $y = f^\Phi(\bar x)$.
  \item Finally, axioms expressing that other atomic formulae are
    interpreted appropriately.
    For example, for an atomic formula
    $P(f(x,g(y)),z)$ we
    need to say that
    $P^\Phi(f^\Phi(x,g^\Phi(y)),z) = P(f(x,g(y)),z)^\Phi$,
    expressed by
    \begin{gather*}
      \inf_{t,w}
      \bigl(
      d(g(y),w)^\Phi
      \vee
      d(f(x,w),t)^\Phi
      \vee
      |P(t,z)^\Phi - P(f(x,g(y)),z)^\Phi|
      \bigr)
      = 0.
    \end{gather*}
  \end{itemize}
  It is relatively straightforward to check that the collection of
  these axioms does define the class $\cK^\Phi$.

  Assume now that $\cK = \Mod(T)$ is an elementary class of
  $\cL_1$-structures and let
  $\Phi^{-1}(T) = T^\Phi \cup \{ \varphi^\Phi \}_{\varphi \in T}$.
  Then $\Phi^{-1}(\cK) = \Mod(\Phi^{-1}(T))$,
  as desired.
\end{proof}

\begin{dfn}[Composition of interpretation schemes, bi-interpretability]
  Assume now that $\Psi\colon \cL_1 \to \cL_2$ is another
  interpretation scheme.
  We then define an interpretation scheme
  $\Psi \circ \Phi\colon \cL_0 \to \cL_2$
  by $\varphi^{\Psi \circ \Phi} = (\varphi^\Psi)^\Phi$
  for each atomic $\cL_2$-formula $\varphi$.
  Again it is straightforward to check that
  $\Phi^{-1}(\cK^\Psi) \subseteq \cK^{\Psi \circ \Phi}$,
  and that if $M \in \Phi^{-1}(\cK^\Psi)$ then
  $\Psi(\Phi(M)) = \Psi \circ \Phi(M)$.

  Finally, consider interpretation schemes
  $\Phi\colon \cL_0 \to \cL_1$
  and
  $\Psi\colon \cL_1 \to \cL_0$,
  and a class $\cK$ of $\cL_0$-structures.
  Assume that $\cK \subseteq \Phi^{-1}(\cK^\Psi)$
  and that
  $\Psi \circ \Phi(M) = M$ (with $\iota = \id_M$)
  for all $M \in \cK$.
  We then say that $\cK$ and
  $\cK' = \Phi(\cK) = \{\Phi(M)\}_{M\in \cK}$
  are (strongly) \emph{bi-interpretable} by $(\Phi,\Psi)$.
  Notice that this is a symmetric notion,
  namely that in this case
  $\cK' \subseteq \Psi^{-1}(\cK^\Phi)$,
  $\cK = \Psi(\cK')$ and
  $\Phi \circ \Psi(N) = N$ for all $N \in \cK'$.
\end{dfn}
Again, our notion of bi-interpretability is stronger than strictly
necessary, and for many applications it suffices to assume that
the mapping $\iota\colon M \to \Psi \circ \Phi(M)$ is uniformly
definable.

\begin{thm}
  \label{thm:BiInterpretableClasses}
  Let $\Phi\colon \cL_0 \to \cL_1$
  and $\Psi\colon \cL_1 \to \cL_0$ be two interpretation schemes,
  and let $\cK$ and $\cK'$ be classes of $\cL_0$- and of
  $\cL_1$-structures, respectively.
  Assume moreover that $\cK$ and $\cK'$ are bi-interpretable
  via $(\Phi,\Psi)$.
  \begin{enumerate}
  \item The class $\cK$ is elementary if and only if $\cK'$ is.
  \item Assume that for each atomic $\cL_1$-formula $\varphi$,
    the definable predicate $\varphi^\Phi$ is constant in $\cK$, and
    similarly that $\varphi^\Psi$ is constant in $\cK'$ for every
    atomic $\cL_0$-formula $\varphi$.
    Then $\cK$ is model complete (respectively, inductive) if and only
    if $\cK'$ is.
  \item Assume that for each atomic $\cL_1$-formula $\varphi$,
    the definable predicate $\varphi^\Phi$ is quantifier-free, and
    similarly that $\varphi^\Psi$ is quantifier-free for every
    atomic $\cL_0$-formula $\varphi$
    (we say that $\Phi$ and $\Psi$ are quantifier-free,
    or that $\cK$ and $\cK'$ are quantifier-free bi-interpretable).
    Then $\cK$ eliminates quantifiers if and only
    if $\cK'$ does.
  \end{enumerate}
\end{thm}
\begin{proof}
  Assume that $\cK = \Mod(T)$.
  Let $T'$ be the theory consisting of
  $\Psi^{-1}(T)$ along with all the axioms of the form
  $\varphi(\bar x) = \varphi^{\Phi \circ \Psi}(\bar x)$,
  where $\varphi$ varies over atomic $\cL_1$-formulae.
  Clearly, if $N \in \cK'$ then $N \models T'$.
  Conversely, assume that $N \models T'$.
  Then $N \models \Psi^{-1}(T)$,
  so $\Psi(N) \in \cK$.
  Thus $\Phi \circ \Psi(N) \in \cK'$, and the second group of axioms
  ensures that $N = \Phi \circ \Psi(N)$.
  Thus $\cK' = \Mod(T')$ is elementary, proving the first item
  (by symmetry).

  For the second item, the assumption tells us that if
  $M_0 \subseteq M_1$ are both in $\cK$ then
  $\Phi(M_0) \subseteq \Phi(M_1)$ in $\cK'$, and similarly in the
  direction from $\cK'$ to $\cK$.
  So assume first that $\cK$ is model complete and let
  $N_0 \subseteq N_1$ in $\cK'$.
  Then $\Psi(N_0) \subseteq \Psi(N_1)$ in $\cK$,
  so $\Psi(N_0) \preceq \Psi(N_1)$
  and thus
  $N_0 = \Phi \circ \Psi(N_0) \preceq \Phi \circ \Psi(N_1) = N_1$.
  Assume now that $\cK$ is inductive and let
  $\{N_i\}_{i<\alpha}$ be an increasing chain in $\cK'$.
  Then
  $\{\Psi(N_i)\}_{i<\alpha}$ is an increasing chain in $\cK$,
  so $M = \bigcup \Psi(N_i) \in \cK$
  (this being a union of complete structures, i.e., the metric
  completion of the set union).
  In particular, $M \supseteq \Psi(N_i)$ for each $i$, so
  $\Phi(M) \supseteq \Phi \circ \Psi(N_i) = N_i$,
  i.e.,
  $\Phi(M) \supseteq \bigcup N_i$.
  We now use the fact that
  the $\cL_0$-definable predicate $d_{\cL_1}(x,y)^\Phi$
  is necessarily uniformly continuous, and that the
  set union of $\Psi(N_i)$ is dense in $M$
  (both with respect to $d_{\cL_0}$)
  to conclude that the set union of the $N_i$ is dense in
  $\Phi(M)$.
  Considering the complete structure union we have
  $\bigcup N_i = \Phi(M) \in \cK'$, as desired.

  We now turn to the last item.
  The assumption tells us that if $\varphi$ is any quantifier-free
  $\cL_0$-formula, or even a quantifier-free $\cL_0$-definable
  predicate, then $\varphi^\Psi$ is quantifier-free as well, and
  similarly in the other direction.
  Assume $\cK$ eliminates quantifiers, and let
  $\varphi(\bar x)$ be and $\cL_1$-formula.
  Then $\varphi^\Phi$ is equivalent in $\cK$ to a quantifier-free
  definable predicate, say $\psi(\bar x)$, and $\psi^\Psi(\bar x)$ is
  quantifier-free as well.
  It will be enough to show that $\psi^\Psi$ coincides with $\varphi$
  in $\cK'$.
  Indeed, let $N = \Phi \circ \Psi(N) \in \cK'$, $\bar a \in N$.
  Then
  \begin{gather*}
    \varphi(\bar a)^N
    = \varphi(\bar a)^{\Phi \circ \Psi(N)}
    = \varphi^\Phi(\bar a)^{\Psi(N)}
    = \psi(\bar a)^{\Psi(N)}
    = \psi^\Psi(\bar a)^N.
  \end{gather*}
  This completes the proof.
\end{proof}

\section{A convergence rate for approximations of the modular
  functional}
\label{apx:ModularApprox}

We conclude with a result that was used in earlier versions of this
paper in \fref{sec:DefinableModular}, later superseded by a more direct
approach.
We chose to keep it here since it is relatively easy
and does provide some uniformity for approximations of Nakano spaces by
ones in which the essential range of $p$ is finite.
Such uniformity may come in handy for an explicit axiomatisation
of Nakano spaces, which, at the time of writing, does not yet exist in
the literature.

A naïve manner to try to approximate the
modular functional is by $\Theta(f) \approx \sum \|f_k\|^{p_k}$
where $f = \sum f_k$ consists of cutting the domain
of $f$ into chunks such that the exponent function $p(\cdot)$ is almost
constant $p_k$ on each chunk.
We show here that these approximations do converge to $\Theta(f)$ at a
uniform rate: the difference is always smaller than
$C\sqrt{\Delta}$ where $\Delta$ is the maximum of diameters of the range of $p$
on the chunks and $C$ is a constant.

\begin{lem}
  \label{lem:LocalModularApproxmation}
  Let $(N,\Theta) = L_{p(\cdot)}(X,\fB,\mu)$, and assume that
  $\essrng p \subseteq [s,s+\varepsilon]$ where
  $1 \leq s < s+\varepsilon \leq r$.
  Let $f \in N$, and assume that $\|f\| \leq 1$.
  Then
  $|\Theta(f)-\|f\|^{s+\varepsilon}|
  \leq \frac{\varepsilon}{s}|\ln \Theta(f)|\Theta(f)$.
\end{lem}
\begin{proof}
  We may assume that $f \geq 0$ and $\|f\| > 0$.
  Let $a = \|f\|$, so $\Theta(f/a) = 1$, and for all $t$:
  \begin{gather*}
    a^t
    = a^t\left(\int (f/a)^p d\mu\right)
    = \int f^p a^{t-p} d\mu,
  \end{gather*}
  Notice that for all $x$ we have
  $s - p(x) \leq 0 \Longrightarrow a^{s-p(x)} \geq 1$
  while
  $s+\varepsilon - p(x)
  \geq 0 \Longrightarrow a^{s+\varepsilon-p(x)} \leq 1$,
  so:
  \begin{gather*}
    a^{s+\varepsilon}
    = \int f^p a^{s+\varepsilon-p} d\mu
    \leq \int f^p d\mu
    \leq \int f^p a^{s-p} d\mu
    = a^s.
  \end{gather*}
  In other words: $a^{s+\varepsilon} \leq \Theta(f) \leq a^s$.
  It follows that
  $\Theta(f)^{1+\frac{\varepsilon}{s}}
  \leq a^{s+\varepsilon}
  \leq \Theta(f)$,
  whereby
  \begin{gather*}
    |\Theta(f) - a^{s+\varepsilon}|
    \leq |\Theta(f) - \Theta(f)^{1+\frac{\varepsilon}{s}}|
    \leq \frac{\varepsilon}{s}|\ln \Theta(f)|\Theta(f),
  \end{gather*}
  as desired.
\end{proof}

\begin{lem}
  \label{lem:SumUpperBound}
  There is a constant $C$ such that for every $0 < n \in \bN$ and
  every sequence $(a_k\colon k < \omega)$ such that $a_k \geq 0$ and
  $\sum a_k \leq 1$:
  \begin{align*}
    &\sum \frac{a_k |\ln a_k|}{k+n}  \leq \frac{C}{\sqrt{n}},
    && (0 \ln 0 = 0).
  \end{align*}
\end{lem}
\begin{proof}
  At first let us assume that $a_k \leq \frac{1}{e}$ for all $k$,
  noting that $\theta(x) = -x\ln x$ is strictly increasing on
  $[0,\frac{1}{e}]$.

  We may assume that the sequence is ordered so that
  $a_k|\ln a_k|$ is decreasing.
  It follows that $(a_k\colon k < \omega)$ is a decreasing sequence.
  Since $\sum a_k \leq 1$ we have
  $a_k \leq \frac{1}{k+1} < \frac{1}{e}$ for all $k \geq 2$, whereby
  $a_k|\ln a_k| \leq  \frac{\ln(k+1)}{k+1} $.
  Let $C_0 = \sum \frac{\ln (k+3)}{(k+3)^{3/2}} < \infty$.
  Then:
  \begin{align*}
    \sum_{k\geq 2} \frac{a_k |\ln a_k|}{k+n}
    & \leq \sum_{k\geq 2} \frac{\ln (k+1)}{(k+n)(k+1)}
    \leq \frac{1}{\sqrt n} \sum_{k\geq 2} \frac{\ln (k+1)}{(k+1)^{3/2}}
    = \frac{C_0}{\sqrt{n}}.
  \end{align*}

  In this calculation we ignored the first two terms of the sum.
  In addition, in the general case
  there may be at most $2$ indexes $k$ such that
  $a_k > \frac{1}{e}$.
  Together these account for at most
  $\frac{4}{en} \leq \frac{4}{e\sqrt n}$.
  Thus
  $\sum \frac{a_k |\ln a_k|}{k+n}  \leq \frac{C}{\sqrt{n}}$ where
  $C = C_0+\frac{4}{e}$.
\end{proof}

\begin{lem}
  \label{lem:LocalModularApproximation}
  Let $(N,\Theta) = L_{p(\cdot)}(X,\fB,\mu)$
  be a Nakano space and let
  $0 <n < \omega$ be fixed.
  Let $\ell > n(r-1)$,
  and for $k < \ell$ let
  $K_k = [\frac{n+k}{n},\frac{n+k+1}{n})$, $X_k = p^{-1}(K_k)$.
  Let $C$ be the constant from \fref{lem:SumUpperBound}.

  Then every $f \in N$ can be expressed as $f = \sum_{k<\ell} f_k$ where
  $f_k  = f\rest_{X_k}
  \in L_{p\rest_{X_k}(\cdot)}(X_k,\fB\rest_{X_k},\mu\rest_{X_k})$.
  If $\|f\| \leq 1$ then we have:
  \begin{gather*}
    \left| \Theta_{p(\cdot)}(f) - \sum_{k<\ell} \|f_k\|^{\frac{n+k+1}{n}} \right|
    \leq \frac{C}{\sqrt n}.
  \end{gather*}
\end{lem}
\begin{proof}
  We have
  $\sum_{k < \ell} \Theta(f_k) = \Theta(f) \leq 1$, whereby
  \begin{align*}
    \left| \Theta(f) - \sum_{k<\ell} \|f_k\|^{\frac{n+k+1}{n}} \right|
    & \leq \sum_{k<\ell} \left| \Theta(f_k) - \|f_k\|^{\frac{n+k+1}{n}} \right| \\
    & \leq \sum_{k<\ell} \frac{1/n}{(n+k)/n} |\ln(\Theta(f_k))|\Theta(f_k) \\
    & = \sum_{k<\ell} \frac{1}{n+k} |\ln(\Theta(f_k))|\Theta(f_k)
    \leq \frac{C}{\sqrt{n}},
  \end{align*}
  as desired.
\end{proof}

\providecommand{\bysame}{\leavevmode\hbox to3em{\hrulefill}\thinspace}
\providecommand{\MR}{\relax\ifhmode\unskip\space\fi MR }
\providecommand{\MRhref}[2]{%
  \href{http://www.ams.org/mathscinet-getitem?mr=#1}{#2}
}
\providecommand{\href}[2]{#2}

\end{document}